\documentclass[journal,twoside,web]{ieeecolor}
\usepackage{generic}
\usepackage{cite}
\usepackage{amsmath,amssymb,amsfonts}
\usepackage{algorithmic}
\usepackage{graphicx}
\usepackage{textcomp}
\usepackage[mathscr]{eucal}
\usepackage{indentfirst}
\usepackage{graphicx}
\usepackage{graphics}
\usepackage{pict2e}
\usepackage{epic}
\usepackage{ulem}
\usepackage{float}
\usepackage{hyperref}

\usepackage{algorithm}

\makeatletter
\def\BState{\State\hskip-\ALG@thistlm}
\makeatother

\newtheorem{theorem}{\bf \emph{Theorem}}[section]
\newtheorem{lemma}[theorem]{\bf \emph{Lemma}}

\newtheorem{proposition}[theorem]{Proposition}

\newtheorem{remark}{Remark}

\newcommand{\transpose}{\mathsf{T}} 

\newcommand{\real}{\mathbb{R}}
\newcommand*{\QEDB}{\hfill\ensuremath{\square}}

\usepackage{tabstackengine}
\setstackEOL{\cr}

\def\BibTeX{{\rm B\kern-.05em{\sc i\kern-.025em b}\kern-.08em
    T\kern-.1667em\lower.7ex\hbox{E}\kern-.125emX}}
\begin{document}
\title{Centralized Versus Decentralized Detection of Attacks in
  Stochastic Interconnected Systems} 

\author{Rajasekhar Anguluri, \IEEEmembership{Student Member, IEEE},
  Vaibhav Katewa, \IEEEmembership{Member, IEEE}, and Fabio Pasqualetti
  \IEEEmembership{Member, IEEE} \thanks{This material is based upon
    work supported by the awards ARO 71603NSYIP, NSF ECCS1405330, and
    UCOP LFR-18-548175. The authors are with the Department of
    Mechanical Engineering, University of California, Riverside,
    \{\href{mailto:rangu003@ucr.edu}{\texttt{rangu003}},\href{mailto:vkatewa@engr.ucr.edu}{\texttt{vkatewa}},\href{mailto:fabiopas@engr.ucr.edu}{\texttt{fabiopas}}\}\texttt{@engr.ucr.edu}.}}
	
\maketitle

\begin{abstract}
  We consider a security problem for interconnected systems governed
  by linear, discrete, time-invariant, stochastic dynamics, where the
  objective is to detect exogenous attacks by processing the measurements
  at different locations. We consider two classes of detectors, namely
  centralized and decentralized detectors, which differ primarily in
  their knowledge of the system model. In particular, a decentralized
  detector has a model of the dynamics of the isolated subsystems, but
  is unaware of the interconnection signals that are exchanged among
  subsystems. Instead, a centralized detector has a model of the
  entire dynamical system. We characterize the performance of the two
  detectors and show that, depending on the system and attack parameters, each of the detectors can outperform the
  other. In particular, it may be possible for the decentralized
  detector to outperform its centralized counterpart, despite having
  less information about the system dynamics, and this surprising
  property is due to the nature of the considered attack detection
  problem. To complement our results on the detection of attacks, we
  propose and solve an optimization problem to design attacks that
  maximally degrade the system performance while maintaining a
  pre-specified degree of detectability. Finally, we validate our
  findings via numerical studies on an electric power system.
\end{abstract}

\section{Introduction}\label{sec:introduction}
Cyber-physical systems are becoming increasingly more complex and
interconnected. In fact, different cyber-physical systems typically
operate in a connected environment, where the performance of each
system is greatly affected by neighboring units. An example is the
smart grid, which arises from the interconnection of smaller power
systems at different geographical locations, and whose performance
depends on other critical infrastructures including the transportation
network and the water system. Given the interconnected nature of large
cyber-physical systems, and the fact that each subsystem usually has
only partial knowledge or measurements of other interconnected units,
the security question arises as to whether sophisticated attackers can
hide their action to the individual subsystems while inducing
system-wide critical perturbations.

In this work we investigate whether, and to what extent, coordination
among different subsystems and knowledge of the global system dynamics
is necessary to detect attacks in interconnected systems. In fact,
while existing approaches for the detection of faults and attacks
typically rely on a centralized detector \cite{fp-fd-fb:10y, YL-AD:16, JC-RJ:99}, the use of local
detectors would not only be computationally convenient, but it would
also prevent the subsystems from disclosing private information about
their plants. As a counterintuitive result, we will show that local
and decentralized detectors can, in some cases, outperform a
centralized detector, thus supporting the development of distributed
and localized theories and tools for the security of cyber-physical
systems.

\noindent
\textbf{Related work:} Centralized attack detectors have been the
subject of extensive research in the last years \cite{YY-QZ-FS-QW-TB, HF-PT-SD, SS-MG, LL-ME-QD-VA-ZH, HZ-PC-LS-JC, YM-BS:16, FM-QZ-MP-GJ, YC-SK-JM, CK-IH:18}, where the detector has complete knowledge of the system dynamics and all
measurements. Furthermore, these studies use techniques from various disciplines including game theory, information theory, fault detection and signal processing, and have a wide variety of applications \cite{YL-AD:16}. Instead, decentralized attack detectors, where each local detector decides on attacks based on partial information and measurements about the system, and local detectors cooperate to improve their detection capabilities, have received only limited and recent attention \cite{fd-fp-fb:11t, NI-IS:14, MN-YY-XW:18, JZ-LM:18, EM-AK-DK:15}. 

Decentralized detection schemes have also been studied for fault detection and isolation (FDI). In such schemes, multiple local detectors make inferences about either the global or local process, and transmit their local decisions to a central entity, which uses appropriate fusion rules to make the global decision\cite{RRT-NRS:81, PK:97, JC-VV;03, SA-VV-DL:05, vr-mm-cg:15}. Methods to improve the detection performance by exchanging information among the local detectors have also been proposed \cite{XG-CE:08, IS-AM-HS-KH:11, RM-TP-MM:12}. These decentralized algorithms are typically complex
\cite{fp-fd-fb:10y}, their effectiveness in detecting unknown and
unmeasurable attacks is difficult to characterize, and their
performance is believed to be inferior when compared to their
centralized counterparts. To the best of our knowledge, a rigorous
comparison of centralized and decentralized attack
detection schemes is still lacking, which prevents us from assessing
whether, and to what extent, decentralized and distributed schemes
should be employed for attack detection and identification.

\noindent
\textbf{Main contributions:}\footnote {In a preliminary version of this paper [26], we used asymptotic approximations to compare the detectors’ performance. Instead, in this paper we provide stronger, tight, and non-asymptotic results without using any approximations. Further, it contains new results on the design of optimal undetectable attacks, and a characterization of the performance degradation induced by such attacks. In addition, an illustration of the results using electrical power grid is also presented.} This paper features three main
contributions. First, we propose centralized and decentralized schemes
to detect unknown and unmeasurable sensor attacks in stochastic
interconnected systems. Our detection schemes are based on the
statistical decision theoretic framework that falls under the category
of simple versus composite hypotheses testing. We characterize the
probability of false alarm and the probability of detection for both
detectors, as a function of the system and attack parameters. Second,
we compare the performance of the centralized and decentralized
detectors, and show that each detector can outperform the other for
certain system and attack configurations. We discuss that this
counterintuitive phenomenon is inherent with the simple versus
composite nature of the considered attack detection problem, and
provide numerical examples of this behavior. Third, we formulate and
solve an optimization problem to design attacks against interconnected
systems that maximally affect the system performance as measured by
the mean square deviation of the state while remaining undetected by
the centralized and decentralized detectors with a pre-selected
probability. Finally, we validate our theoretical findings on the IEEE
RTS-96 power system model.

\noindent
\textbf{Paper organization:} The rest of the paper is organized as follows. Section \ref{sec: problem setup and preiliminary notions} contains our problem formulation. In Section \ref{sec: detection
	framework}, we present our local, decentralized, and centralize detectors, and characterize their performance. Section \ref{sec:
	CDA} contains our main results regarding the comparison of the performance of centralized and decentralized detectors. Section \ref{sec: DOOA} contains the design of optimal undetectable attacks. Finally, Section \ref{sec: simulations} contains our numerical studies, and Section \ref{sec: conclusions} concludes the paper. 
%

\noindent
\textbf{Mathematical notation:} The following notation will be adopted
throughout the paper. Let $X_1,\ldots,X_N$ be arbitrary sets, then
$\bigcup_{i=1}^N X_i$ and $\bigcap_{i=1}^N X_i$ denotes the union and
intersection of the sets, respectively.  $\mathrm{Trace} (\cdot)$,
$\mathrm{Rank}(\cdot)$, and $\mathrm{Null}(\cdot)$ denote the trace,
rank, and null space of a matrix, respectively. $Q>0$ ($Q\geq0$)
denotes that $Q$ is a positive definite (positive semi definite)
matrix. $\otimes$ denotes the Kronecker product for
matrices. $\mathrm{blkdiag}(A_1, \cdots , A_N)$ denotes the block
diagonal matrix with $A_1, \cdots, A_N$ as diagonal entries. The
identity matrix is denoted by $I$ (or $I_\text{dim}$ to denote
dimension explicitly). $\mathrm{Pr}[\mathcal{E}]$ denotes the
probability of the event $\mathcal{E}$. The mean and covariance of a
real or vector valued random variable $Y$ is denoted by
$\mathbb{E}[Y]$ and $\mathrm{Cov}[Y]$. Further, for a real valued
random variable $Y$, we denote the standard deviation as
$\mathrm{SD}[Y]$. If $Y$ follows a Gaussian distribution, we denote it
by $Y\sim\mathcal{N}\left(\mathbb{E}[Y],
  \mathrm{Cov}[Y]\right)$. Instead, if $Y$ follows a noncentral
chi-squared distribution, we denote it by $Y\sim \chi^2(p,\lambda)$,
where $p$ is the degrees of freedom and $\lambda$ is the
non-centrality parameter. For $Y \sim \chi^2(p,\lambda)$ and $\tau \geq 0$, $Q(\tau;p,\lambda)$ denotes the
complementary cumulative distribution function of $Y$.

\section{Problem setup and preliminary notions}\label{sec: problem setup and preiliminary notions}
We consider an interconnected system with $N$ subsystems, where each
subsystem obeys the discrete-time linear dynamics
\begin{align}\label{eq: subystem without attack}
  \begin{split}
    x_i(k+1)&=A_{ii}x_i(k) + B_iu_i(k)+ w_i(k),\\
    y_i(k)&=C_ix_i(k)+v_i(k), 
  \end{split}
\end{align}
with $i \in \{1,\ldots,N\}$. In the above equation, the vectors
$x_i \in \mathbb{R}^{n_i}$ and $y_i \in \mathbb{R}^{r_i} $ are the
state and measurement of the $i-$th subsystem, respectively. The
process noise $w_i(k)\sim \mathcal{N}(0, \Sigma_{w_i})$ and the
measurement noise $v_i(k)\sim \mathcal{N}(0,\Sigma_{v_i})$ are
independent stochastic processes, and $w_i$ is assumed to be
independent of $v_i$, for all $k\geq 0$. Further, the noise vectors
across different subsystems are assumed to be independent at all
times. The $i-$th subsystem is coupled with the other subsystems
through the term $B_iu_i$, which takes the form
\begin{align*}
  B_i&=
       \begin{bmatrix}
         A_{i1} & \cdots & A_{i,i-1} & A_{i,i+1} & \cdots & A_{iN}
       \end{bmatrix}
                                                            ,
                                                            \text{ and}\\
  u_i
     &=
       \begin{bmatrix}
         x_1^\transpose & \cdots & x_{i-1}^\transpose &
         x_{i+1}^\transpose &  \cdots & x_N^\transpose
       \end{bmatrix}^\transpose.
\end{align*}
 The input
$B_i u_i = \sum_{j \neq i}^N A_{ij}x_j$ represents the cumulative
effect of subsystems $j$ on subsystem $i$. Hence, we refer to $B_i$ as
to the interconnection matrix, and to $u_i$ as to the interconnection
signal, respectively.

We allow for the presence of attacks compromising the dynamics of the
subsystems, and model such attacks as exogenous unknown inputs. In
particular, the dynamics of the $i-$th subsystem under the attack
$u^a_i$ with matrix $B^a_i$ read as
\begin{align}\label{eq: subsystem with attack}
  \begin{split}
    x_i(k+1)&=A_{ii}x_i(k)+B_iu_i(k)+B^a_iu^a_i(k)+w_i(k), 
  \end{split}
\end{align}
where $u_i^a \in
\real^{m_i}$. In vector form, the dynamics of the interconnected
system under attack read as
\begin{align}\label{eq: overall with attack}
\begin{split}
 x(k+1)&=A x(k) + B^a u^a(k) + w(k),\\
y(k)&=Cx(k)+v(k),
\end{split}
\end{align}
where
$\phi=\begin{bmatrix} \phi_1^\transpose & \ldots &\phi_N^\transpose
 \end{bmatrix}$, with $\phi$ standing for $x\in \mathbb{R}^n$,
$w \in \mathbb{R}^n$, $u^a \in \mathbb{R}^m$, $y \in \mathbb{R}^r$,
$v \in \mathbb{R}^r$, $n=\sum_{i=1}^{N}n_i$, $m = \sum_{i =1}^N m_i$, and
$r=\sum_{i}^{N}r_i$. Moreover, as the components of the vectors $w$
and $v$ are independent and Gaussian,
$w \sim \mathcal{N}(0,\Sigma_{w})$ and
$v \sim \mathcal{N}(0,\Sigma_{v})$, respectively, where
$\Sigma_w=\mathrm{blkdiag}\left(\Sigma_{w_1},\ldots, \Sigma_{w_N}
\right)$ and
$\Sigma_v=\mathrm{blkdiag}\left(\Sigma_{v_1},\ldots, \Sigma_{v_N}
\right)$. Further,
\begin{align*}
  A = 
  \begin{bmatrix}
    A_{11} & \cdots & A_{1N}\\
    \vdots & \ddots & \vdots\\
    A_{N1} & \cdots & A_{NN}\\
  \end{bmatrix}, 
  B^a=\begin{bmatrix}
  B_{1}^a & \cdots & 0\\
  \vdots & \ddots & \vdots\\
  0 & \cdots & B^a_{N}\\
  \end{bmatrix}, 
\end{align*}
and $C=\mathrm{blkdiag}\left(C_1,\ldots,C_N\right)$.
 
We assume that each subsystem is equipped with a \textit{local
  detector}, which uses the local measurements and knowledge of the
local dynamics to detect the presence of local attacks. In particular,
the $i-$th local detector has access to the measurements $y_i$ in
\eqref{eq: subystem without attack}, knows the matrices $A_{ii}$, $B_i$,
and $C_i$, and the statistical properties of the noise vectors $w_i$
and $v_i$. Yet, the $i-$th local detector does not know or measure the
interconnection input $u_i$, and the attack parameters $B_i^a$ and
$u_i^a$. Based on this information, the $i-$th local detector aims to
detect whether $B_i^a u_i^a \neq 0$. The decisions of the local
detectors are then processed by a \textit{decentralized detector},
which aims to detect the presence of attacks against the whole
interconnected system based on the local decisions. Finally, we assume
the presence of a \textit{centralized detector}, which has access to
the measurements $y$ in \eqref{eq: overall with attack}, and knows the
matrix $A$ and the statistical properties of the overall noise vectors
$w$ and $v$. Similarly to the local detectors, the centralized detector
does not know or measure the attack parameters $B^a$ and $u^a$, and
aims to detect whether $B^a u^a \neq 0$. We postpone a detailed
description of our detectors to Section \ref{sec: detection
  framework}. To conclude this section, note that the decentralized
and centralized detectors have access to the same measurements. Yet,
these detectors differ in their knowledge of the system dynamics,
which determines their performance as explained in Section~\ref{sec:
  CDA}.

\begin{remark}{\it \textbf{(Control input and initial state)}}\label{rmk: control and initial state}
  The system setup in \eqref{eq: subsystem with attack} and \eqref{eq:
    overall with attack} typically includes a control input. However,
  assuming that each subsystem knows its control input, it can be omitted without affecting generality. Further, as the
  detectors do not have information about the initial state, we
  assume without loss of generality, that the initial state is
  deterministic and unknown to the detectors. \QEDB
\end{remark}

\section{Local, decentralized, and centralized detectors}\label{sec:
  detection framework}
In this section we formally describe our local, decentralized, and
centralized detectors, and characterize their performance as a
function of the available measurements and knowledge of the system
dynamics. To this aim, let $T>0$ be an arbitrary time horizon and define the vectors
\begin{align}\label{eq: agree local}
  Y_i =
  \begin{bmatrix}
    y_i^\transpose (1) & y_i^\transpose (2)  & \cdots & y_i^\transpose (T)
  \end{bmatrix}^\transpose,
\end{align}
which contains the measurements available to the $i-$th detector, and
\begin{align}\label{eq: agree central}
  Y_c =
\begin{bmatrix}
y^\transpose (1) & y^\transpose (2)  & \cdots & y^\transpose (T)
\end{bmatrix}^\transpose,
\end{align}
which contains the measurements available to the centralized
detector. Both the local and centralized detectors perform the following
three operations in order:
\begin{enumerate}
\item Collect measurements as in \eqref{eq: agree local} and
  \eqref{eq: agree central}, respectively;

\item Process measurements to filter unknown variables; and

\item Perform statistical hypotheses testing to detect attacks
  (locally or globally) using the processed measurements.
\end{enumerate}
The decisions of the local detectors are then used by the
decentralized detector, which triggers an alarm if any of the local
detectors does so. We next characterize how the detectors process
their measurements and perform attack detection via statistical
hypothesis testing.

\subsection{Processing of measurements}\label{subsec: pom}
The measurements \eqref{eq: agree local} and \eqref{eq: agree central} depend on parameters that are unknown to the detectors,
namely, the system initial state and the
interconnection signal (although the process and measurement noises
are also unknown, the detectors know their statistical
properties). Thus, to test for the presence of attacks, the detectors
first process the measurement vectors to eliminate their dependency on
the unknown parameters. To do so, using equations \eqref{eq: subystem
  without attack} and \eqref{eq: subsystem with attack}, define the
observability matrix and the attack, interconnection, and noise forced
response matrices of the $i-$th subsystem as
\begin{align*}
  \mathcal{O}_i &=
                            \begin{bmatrix}
                              C_iA_{ii} \\ \vdots \\ C_i A_{ii}^{T}
                            \end{bmatrix}, \;
  \mathcal{F}^{(a)}_i=
  \begin{bmatrix}
    C_iB^a_i & \ldots & 0\\
    \vdots & \ddots & \vdots\\
    C_iA_{ii}^{T-1}B^a_i & \ldots & C_iB^a_i
  \end{bmatrix}, 
\end{align*}
\begin{align*}
  \mathcal{F}_{i}^{(u)}\begingroup
  \setlength\arraycolsep{1.2pt}=
  \begin{bmatrix}
    C_iB_i & \ldots & 0\\
    \vdots & \ddots & \vdots\\
    C_iA_{ii}^{T-1}B_i  & \ldots & C_iB_i
  \end{bmatrix},
                                   \endgroup
                                   \;
                                   \mathcal{F}^{(w)}_i\begingroup
                                   \setlength\arraycolsep{1.6pt}=\begin{bmatrix}
                                     C_i & \ldots & 0\\
                                     \vdots  & \ddots & \vdots\\
                                     C_iA_{ii}^{T-1} &  \ldots & C_i
                                   \end{bmatrix}\endgroup.                                         
\end{align*}
Analogously, for the system model \eqref{eq: overall with attack}
define the matrices $\mathcal{O}_{c}$, $\mathcal{F}^{(w)}_{a}$, and
$\mathcal{F}^{(w)}_{c}$, which are constructed as above by replacing
$A_{i}$, $B_i^a$, and $C_i$ with $A$, $B^a$, and $C$,
respectively. The measurements \eqref{eq: agree local} and \eqref{eq:
  agree central} can be written as follows:
\begin{align}
  Y_i&=\mathcal{O}_ix_i(0)+\mathcal{F}_{i}^{(u)}U_i +
       \mathcal{F}_i^{(a)} U^a_i + \mathcal{F}_i^{(w)}W_i+V_i , \label{eq: expanded measurements local}\\
  Y_c&=\mathcal{O}_{c}x(0)+\mathcal{F}^{(a)}_{c}U^a +
       \mathcal{F}_{c}^{(w)}W+V, \label{eq: expanded measurements central}
\end{align}
where
$U_i= \begin{bmatrix} u_i^\transpose(0) & u_i^\transpose(1) & 
  \cdots & u_i^\transpose(T-1)
\end{bmatrix}^\transpose$. The vectors $U^a_i$, $U^a$, $W_i$ and $W$
are the time aggregated signals of $u^a_i$, $u^a$, $w_i$, and $w$,
respectively, and are defined similarly to $U_i$. Instead,
$V_i=\begin{bmatrix} v_i^\transpose(1)& v_i^\transpose(2) & \cdots &
  v_i^\transpose(T)
\end{bmatrix}^\transpose$, and $V$ is defined similarly to $V_i$. To
eliminate the dependency from the unknown variables, let $N_i$ and
$N_{c}$ be bases of the left null spaces of the matrices
$\begin{bmatrix} \mathcal{O}_i & \mathcal{F}^{(u)}_i \end{bmatrix}$
and $\mathcal{O}_{c}$, respectively, and define the processed
measurements as
\begin{align}\label{eq: kernels}
  \begin{split}
    \widetilde Y_i&=N_i Y_i= N_i\left[\mathcal{F}_i^{(a)}U^a_i + \mathcal{F}_i^{(w)}W_i+V_i\right],\\
    \widetilde Y_{c}&= N_{c}
    Y_{c}=N_{c}\left[\mathcal{F}^{(a)}_{c}U^a +
      \mathcal{F}_{c}^{(w)}W+V\right] ,
  \end{split}
\end{align}
where the expressions for $\widetilde Y_i$ and $\widetilde Y_{c}$
follows from \eqref{eq: expanded measurements local} and \eqref{eq: expanded measurements central}. Notice that, in
the absence of attacks ($U^a=0$), the measurements $\widetilde Y_{i}$
and $\widetilde Y_{c}$ depend only on the system noise. Instead, in
the presence of attacks, such measurements also depend on the attack
vector, which may leave a signature for the detectors.\footnote{If
  $\mathrm{Im}({B}_i^a)\subseteq \mathrm{Im}({B}_i)$, then
  $N_i \mathcal{F}^{(a)}_i=0$ and the processed measurements do not
  depend on the attack. Thus, our local detection technique can only
  be successful against attacks that do not satisfy this condition.}
We now characterize the statistical properties of $\widetilde Y_i$ and
$\widetilde Y_{c}$.

\begin{lemma} {\it \textbf{(Statistical properties of the processed measurements)}}\label{thm: stat prop of processed}
  The processed measurements $\widetilde Y_i$ and $\widetilde Y_c$
  satisfy
  \begin{align}\label{eq: mean and covariance}
    \begin{split}
      \widetilde Y_i & \sim \mathcal{N}\left(\beta_i,\Sigma_i\right),
      \text{ for all } i \in \{1,\ldots,N\}, \text{ and } \\
      \widetilde Y_c & \sim \mathcal{N}\left(\beta_c,\Sigma_c\right),
    \end{split}
  \end{align}
  where
  \begin{align}\label{eq: beta and sigma original}
  \begin{split}
  \beta_i&=N_i\mathcal{F}^{(a)}_iU_i^a,\\
\beta_c&= N_c\mathcal{F}^{(a)}_cU^a,\\
\Sigma_i&= N_i\left[ \left( \mathcal{F}^{(w)}_i\right) \left(I_T\otimes \Sigma_{w_i}\right) \left( \mathcal{F}^{(w)}_i\right)^\transpose+\left(I_T\otimes \Sigma_{v_i}\right)\right] N_i^\transpose,\\
\Sigma_c&= N_c\left[\left( \mathcal{F}^{(w)}_c\right) \left(I_T\otimes
\Sigma_{w}\right) \left(
\mathcal{F}^{(w)}_c\right)^\transpose+\left(I_T\otimes
\Sigma_{v}\right)\right]N_c^\transpose.
  \end{split}
  \end{align}
\end{lemma}

\medskip A proof of Lemma \ref{thm: stat prop of processed} is
postponed to the Appendix. From Lemma \ref{thm: stat prop of
  processed}, the mean vectors $\beta_i$ and $\beta_c$ depend on the
attack vector, while the covariance matrices $\Sigma_i$ and $\Sigma_c$
are independent of the attack. This observation motivates us to
develop a detection mechanism based on the mean of the processed
measurements, rather the covariance matrices.

\subsection{Statistical hypothesis testing framework}\label{sec: SHT}
In this section we detail our attack detection mechanism, which we
assume to be the same for all local and centralized detectors, and we
characterize its false alarm and detection probabilities. We start by
analyzing the test procedure of the $i-$th local detector. Let $H_0$
be the null hypothesis, where $\beta_i = 0$ and the system is not
under attack, and let $H_1$ be the alternative hypothesis, where
$\beta_i \neq 0$ and the system is under attack. To decide which
hypothesis is true, or equivalently whether the mean value of the
processed measurements is zero, we resort to the generalized
log-likelihood ratio test (GLRT):
\begin{align}\label{eq: test_stat_local}
  \Lambda_i\triangleq {\widetilde  Y_i}^\transpose \Sigma_i^{-1}{
  \widetilde  Y_i}\overset{H_1}{\underset{H_0}{\gtrless}} \tau_i ,
\end{align}
where the threshold $\tau_i \geq 0$ is selected based on the desired
false alarm probability of the test \eqref{eq: test_stat_local}
\cite{HVP:94}. For a statistical hypothesis testing problem, the false
alarm probability equals the probability of deciding for $H_1$ when
$H_0$ is true, while the detection probability equals the probability
of deciding for $H_1$ when $H_1$ is true. While the former is used for
tuning the threshold, the latter is used for measuring the performance
of the test. Formally, the false alarm and detection probabilities of
\eqref{eq: test_stat_local} are the probabilities that are conditioned
on the hypothesis $H_0$ and $H_1$, respectively, and are
symbolically denoted as
\begin{align*}
  P^F_i=\mathrm{Pr}\left[\Lambda_i\geq \tau_i| H_0\right]
  \text{ and } P^D_i= \mathrm{Pr}\left[ \Lambda_i\geq \tau_i |
  H_1\right].
\end{align*}
Similarly, the centralized detector test is defined as
\begin{align}\label{eq: test_stat_central}
  \Lambda_c\triangleq {\widetilde
  Y_c}^\transpose\Sigma_c^{-1}{\widetilde
  Y_c}\overset{H_1}{\underset{H_0}{\gtrless}} \tau_c ,
\end{align}
where $\tau_c \geq 0$ is a preselected threshold, and its false alarm
and detection probabilities are denoted as $P^F_c$ and $P^D_c$. We
next characterize the false alarm and detection probabilities of the
detectors with respect to the system and attack parameters.

\begin{lemma}{\it \textbf{(False alarm and detection probabilities of
      local and centralized detectors)}}\label{lma: PF and PD via Q-functions}
  The false alarm and the detection probabilities of the tests
  \eqref{eq: test_stat_local} and \eqref{eq: test_stat_central} are,
  respectively,
  \begin{align}\label{eq: test probabilities}
    \begin{split}
      P^F_i&= Q(\tau_i;p_i,0), \; P^D_i = Q(\tau_i;p_i,\lambda_i),
      \text{
        and }\\
      P^F_c&=Q(\tau_c;p_c,0), \; P^D_c = Q(\tau_c;p_c,\lambda_c) ,
    \end{split}
  \end{align}
  where 
  \begin{align}\label{eq: lambda 1}
    \begin{split}
      p_i &= \mathrm{Rank}(\Sigma_i), \; p_c=\mathrm{Rank}(\Sigma_c),\\
      \lambda_i &=(U_i^a)^\transpose M_i(U_i^a), \;
      \lambda_c=(U^a)^\transpose M_c(U^a) ,
    \end{split}
  \end{align}
  and
  \begin{align}\label{eq: Mi and Mc}
    \begin{split}
      M_i &= \left(
        N_i\mathcal{F}^{(a)}_i\right)^\transpose \Sigma_i^{-1}\left(
       N_i\mathcal{F}^{(a)}_i\right),\\
      M_c &= \left( N_c\mathcal{F}^{(a)}_c\right)
      ^\transpose\Sigma_c^{-1}\left(
        N_c\mathcal{F}^{(a)}_c\right).
    \end{split}
  \end{align}
\end{lemma}

\medskip Lemma \ref{lma: PF and PD via Q-functions}, whose proof is
postponed to the Appendix, allows us to compute the false alarm and
detection probabilities of the detectors using the decision
thresholds, the system parameters, and the attack vector. Moreover,
for fixed $P^F_i$ and $P^F_c$, the detection thresholds are computed
as $\tau_c = Q^{-1}(P^F_c;p_c,0)$ and $\tau_i = Q^{-1}(P^F_i;p_i,0)$,
where $Q^{-1}(\cdot)$ is the inverse of the complementary Cumulative
Distribution Functions (CDF) that is associated with a central
chi-squared distribution. The parameters $p_i$, $p_c$ and $\lambda_i$,
$\lambda_c$ in Lemma \ref{lma: PF and PD via Q-functions} are referred
to as \textit{degrees of freedom} and \textit{non-centrality}
parameters of the detectors.

\begin{remark}{\it {\textbf{(System theoretic interpretation of detection
      probability parameters)}}}\label{rmk: sys theor interp}
  The degrees of freedom and the non-centrality parameters quantify
  the knowledge of the detectors about the system dynamics and the
  energy of the attack signal contained in the processed
  measurements. In particular:

  \smallskip
  \noindent
  \textit{(Degrees of freedom $p_i$)} The detection probability and
  the false alarm probability are both increasing functions of the
  degrees of freedom $p_i$, because the $Q$ function in \eqref{eq:
    test probabilities} is an increasing function of $p_i$. Thus,
  increasing $p_i$ by, for instance, increasing the number of sensors
  or the horizon $T$, does not necessarily lead to an improvement of
  the detector performance.

  
  \smallskip
  \noindent
  \textit{(Non-centrality parameter $\lambda_i$)} The non-centrality
  parameter $\lambda_i$ measures the energy of the attack signal
  contained in the processed measurements. In the literature of
  communication and signal processing, the non-centrality parameter is
  often referred to as signal to noise ratio (SNR)
  \cite{HVP:94}. For fixed
  $\tau_i$ and $p_i$, the detection probability increases
  monotonically with $\lambda_i$, and approaches the false alarm
  probability as $\lambda_i$ tends to zero. 


  \smallskip
  \noindent
  \textit{(Decision threshold $\tau_i$)} For fixed $\lambda_i$ and
  $p_i$, the probability of detection and the false alarm probability
  are monotonically decreasing functions of the detection threshold
  $\tau_i$. This is due to the fact that the complementary CDFs, which
  define the false alarm and detection probabilities, are decreasing
  functions of $\tau_i$. As we show later, because of the contrasting
  behaviors of the false alarm and detection probabilities with
  respect to all individual parameters, the decentralized detector can
  outperform the centralized detector. \QEDB
\end{remark}

We now state a result that provides a relation between the degrees of
freedom ($p_i$ and $p_c$) and the non-centrality parameters
($\lambda_i$ and $\lambda_c$) of the local and the centralized detectors. This result plays a central role in
comparing the performance of these centralized and decentralized
detectors.

\begin{lemma} {\it \textbf{(Degrees of freedom and non-centrality 
      parameters)}}\label{lma: parameters} Let $p_i$, $p_c$ and
  $\lambda_i$, $\lambda_c$ be the degrees of freedom and
  non-centrality parameters of the $i-$th local and centralized detectors,
  respectively. Then, $p_i\!\leq\!p_c$ and
  $\lambda_i\!\leq \!\lambda_c$ for all~$i \in \{1,\dots,N\}$.
\end{lemma}

A proof of Lemma \ref{lma: parameters} is postponed to the Appendix.
In loose words, given the interpretation of the degrees of freedom and
noncentrality parameters in Remark \ref{rmk: sys theor interp}, Lemma
\eqref{lma: parameters} states that a centralized detector has more
knowledge about the system dynamics ($p_i \le p_c$) and its
measurements contain a stronger attack signature
($\lambda_i \le \lambda_c$) than any of the $i-$th local
detector. Despite these properties, we will show that the
decentralized detector can outperform the centralized one.

\section{Comparison of centralized and decentralized detection of
  attacks}\label{sec: CDA}
In this section we characterize the detection probabilities of the
decentralized and centralized detectors, and we derive sufficient
conditions for each detector to outperform the other. Recall that the
decentralized detector triggers an alarm if any of the local detectors
detects an alarm. In other words, 
\begin{align}\label{eq: PFD PDD definitions}
  \begin{split}
    P^F_d &= \mathrm{Pr}\left[ \Lambda_i \ge \tau_i, \text{ for some }
      i\in\{1,\ldots,N\} \left.\right\vert
      H_0\right],
    \\
    P^D_d &= \mathrm{Pr}\left[ \Lambda_i \ge \tau_i, \text{ for some }
      i\in\{1,\ldots,N\} \left.\right\vert
      H_1\right],
  \end{split}
\end{align}
where $P^F_d$ and $P^D_d$ denote the false alarm and detection
probabilities of the decentralized detector, respectively.

\begin{lemma} {\it \textbf{(Performance of the decentralized
      detector)}}\label{lma: PFD and PDD} The false alarm and
  detection probabilities in \eqref{eq: PFD PDD definitions} satisfy
  \begin{align} \label{eq: PFd PDd expressions}
    P^F_d=1-\prod\limits_{i=1}^{N}\left(1-P^F_i\right), \text{ and }
    P^D_d=1-\prod\limits_{i=1}^{N}\left(1-P^D_i\right).
  \end{align}	
\end{lemma}

A proof of Lemma \ref{lma: PFD and PDD} is postponed to the
Appendix. As shown in Fig. \ref{fig: false_alarm_behaviors}, for the case when $P^F_i=P^F_j$, for all $i,j\in \{1,\ldots,N\}$, $P^F_d$
increases with increase in $P^F_i$ and
$N$. To allow for a fair comparison between the decentralized and
centralized detectors, we assume that $P_c^F = P^F_d$. Consequently,
for a fixed false alarm probability $P_c^F$, the probabilities $P_i^F$
satisfy
\begin{align*}
  P_c^F  = 1-\prod\limits_{i=1}^{N}\left(1-P^F_i\right) .
\end{align*}

We now derive a sufficient condition for the centralized
detector to outperform the decentralized detector.

\begin{figure}[t]
	\centering
	\includegraphics[width=1.0\linewidth]{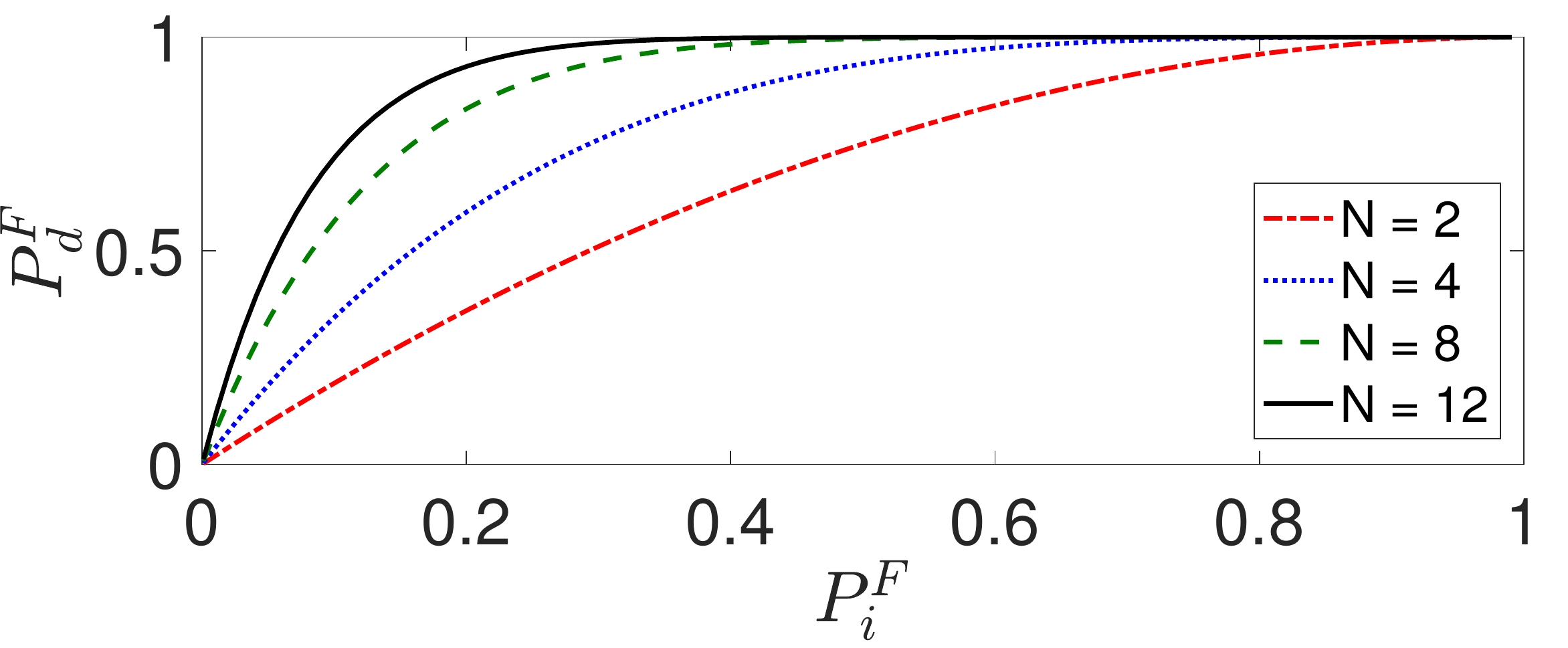}
	\caption {This figure shows the false alarm probability of the
          decentralized detector, $P_d^F$, as a function of the identical false
          alarm probabilities of the local detectors, $P_i^F$, for
          different numbers of local detectors.}
	\label{fig: false_alarm_behaviors}
\end{figure}

\begin{theorem}{\it \textbf{(Sufficient condition for
      $P^D_c\geq P^D_d$)}}\label{thm: Pc greater than Pd}
  Let $P^F_c = P_d^F$, and assume that the following condition is satisfied:
  \begin{align}\label{eq: PDc greater than PDd}
    \tau_c \le p_c + \lambda_c - \sqrt{4N (p_c + 2\lambda_c)
    \ln\left( \frac{1}{1-P^D_\text{max}}\right) },
  \end{align}
  where $P^D_\text{max} = \max \{ P_1^D, \dots, P_N^D\}$. Then,
  $P^D_c\geq P^D_d$.
\end{theorem}

A proof of Theorem \ref{thm: Pc greater than Pd} is postponed to the
Appendix. We next derive a sufficient condition for the decentralized
detector to outperform the centralized detector.

\begin{theorem}{\it \textbf{(Sufficient condition for
      $P^D_d\geq P^D_c$)}}\label{thm: Pd greater than Pc}
  Let $P^F_c = P_d^F$, and assume that the following condition is satisfied:
  \begin{align}\label{eq: PDd greater than PDc}
  \begin{split}
    \tau_c & \geq p_c+\lambda_c+\sqrt{4\left(p_c+2\lambda_c\right) \ln
     \left(\frac{1}{1-(1-P^D_\text{min})^N}\right)} + \\
    &\quad\quad+2\ln\left(\frac{1}{1-(1-P^D_\text{min})^N} \right) ,
  \end{split}
  \end{align}
 where $P^D_\text{min} = \min \{ P_1^D, \dots, P_N^D\}$. Then $P^D_d\geq P^D_c$.
\end{theorem}

\smallskip A proof of Theorem \ref{thm: Pc greater than Pd} is
postponed to the Appendix. Theorems \ref{thm: Pc greater than Pd} and
\ref{thm: Pd greater than Pc} provide sufficient conditions on the
detectors and attack parameters that result in one detector outperforming the other. In particular, from
\eqref{eq: PDc greater than PDd} and \eqref{eq: PDd greater than PDc}
we note that, depending on decision threshold $\tau_c$, a centralized
detector may or may not outperform a decentralized detector. This is
intuitive as the $Q$ function, which quantifies the detection
probability, is a decreasing function of the detection threshold (see
Remark \ref{rmk: sys theor interp}).  To clarify the effect of attack
and detection parameters on the performance trade-offs of the
detectors, we now express \eqref{eq: PDc greater than PDd} and
\eqref{eq: PDd greater than PDc} using the mean and standard deviation
of the test statistic $\Lambda_c$ in \eqref{eq:
  test_stat_central}. Let
\begin{align*}
\begin{split}
  \mu_c&\triangleq \mathbb{E}\left[\Lambda_c \right] = \lambda_c+p_c,
  \text{ and }\\
  \sigma_c& \triangleq \text{SD}[\Lambda_c] =
  \sqrt{2(p_c+2\lambda_c)}.
\end{split} 
\end{align*}
where the expectation and standard deviation (SD) of $\Lambda_c$ follows from the fact that
under $H_1$, $\Lambda_c \sim \chi^2(p_c,\lambda_c)$ (see proof of Lemma \ref{lma: PF and PD via Q-functions}).
Hence, \eqref{eq: PDc greater than PDd} and \eqref{eq: PDd greater
  than PDc} can be rewritten, respectively, as
\begin{subequations}
\begin{align}
  \tau_c&\leq \mu_c-\sigma_c \underbrace{\sqrt{2
          N\ln\left( \frac{1}{1-P^D_\text{max}}\right) }}_{\triangleq \kappa_c} ,\text{ and } \label{eq: modified sufficient condition 1}\\
  \tau_c&\geq \mu_c+\sigma_c \underbrace{\sqrt{2
          \ln\left(
          \frac{1}{1-(1-P^D_\text{min})^N}\right)}}_{\triangleq
          \kappa_d}+\kappa_d^2\label{eq: modified sufficient condition
          2} .
\end{align}
\end{subequations}
From \eqref{eq: modified sufficient condition 1} and \eqref{eq:
  modified sufficient condition 2} we note that a centralized detector
outperforms the decentralized one if $\tau_c$ is $\kappa_c$ standard
deviations smaller than the mean $\mu_c$. Instead, a
decentralized detector outperforms the centralized detector if
$\tau_c$ is at least $\kappa_d$ standard deviations larger than the mean $\mu_c$. See Fig. \ref{fig: pdf of central
  detector} for a graphical illustration of this interpretation.
\begin{figure}[t]
	\centering
	\includegraphics[width=1.0\linewidth]{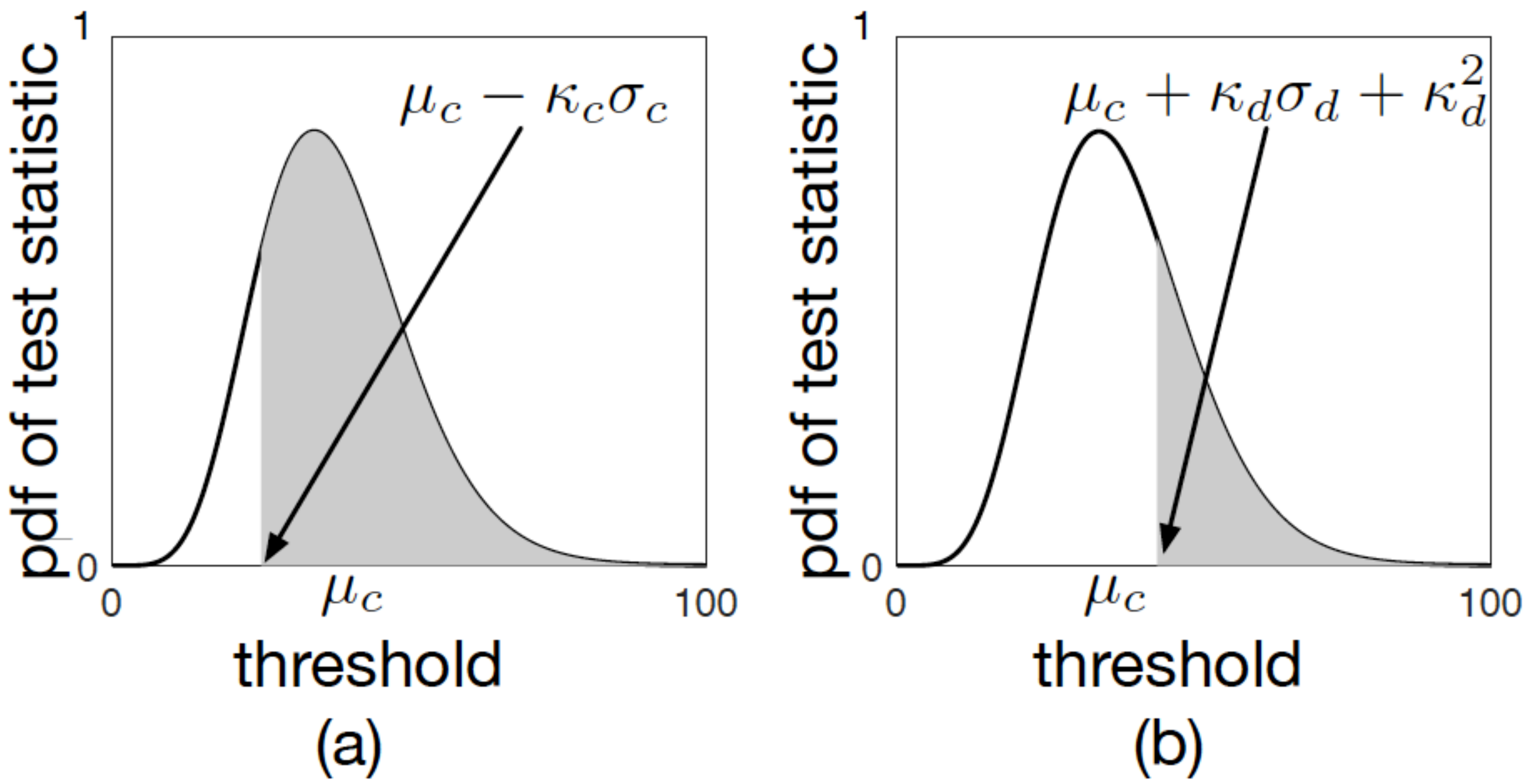}
	\caption {This figure shows the probability density function (pdf)
		of $\Lambda_c$ under $H_1$, as a function of threshold $\tau_c$. For $\tau_c=\mu_c-\kappa_c\sigma_c$ and $\tau_c=\mu_c+\kappa_d\sigma_d+\sigma_d^2$, the shaded
		area in panels (a) and (b)  indicates the detection probability of the centralized detector. As seen in panels (a) and (b), an increase in $\kappa_c$ results in larger area (larger detection probability) while a increase in $\kappa_d$ results in smaller area (smaller detection probability).}
	\label{fig: pdf of central detector}
\end{figure}

Theorems \ref{thm: Pc greater than Pd} and \ref{thm: Pd greater than
  Pc} are illustrated in Fig. \ref{fig: non-centrality-regions} as a
function of the non-centrality parameters. It can be observed that (i)
each of the detectors can outperform the other depending on the values of the noncentrality parameter values, (ii) the provided
bounds qualitatively capture the actual performance of the centralized
and decentralized detectors as the non-centrality parameters increase,
and (iii) the provided bounds are rather tight over a large range of
non-centrality parameters. In Fig.~\ref{fig: PDC minus PDD} we show
that the difference of the detection probabilities of the centralized
and decentralized detectors can be large, especially when the
non-centrality parameters are small and satisfy
$\lambda_c \approx \lambda_i$, as evident in panel (a) of Fig.~\ref{fig: PDC minus PDD} .

\begin{figure}[t]
	\centering
	\includegraphics[width=1.0\linewidth]{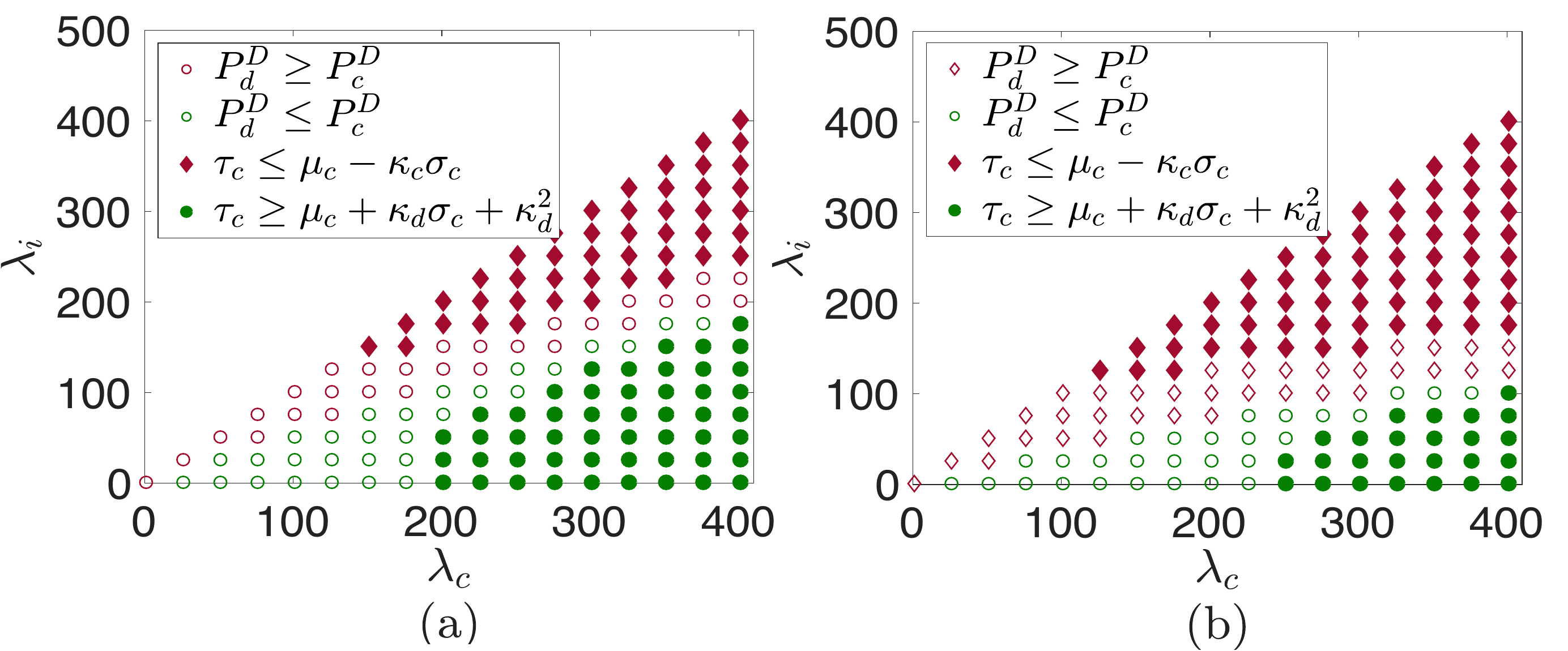}
	\caption{This figure shows when the decentralized, which comprises identical local detectors, and
          centralized detectors outperform their counterpart, as a
          function of the non-centrality parameters. The regions
          identified by solid markers correspond to the conditions in
          Theorems \ref{thm: Pc greater than Pd} and \ref{thm: Pd
            greater than Pc}. Instead, regions identified by empty
          markers are identified numerically. Since
          $\lambda_i \leq \lambda_c$, the white region (top left) is
          not admissible. For a fixed $P^F_c =P^F_d =0.01$, (a)
          corresponds to the case of $N=2$ and (b) corresponds to
          the case of $N=4$. When $N=4$, the decentralized detector outperforms the centralized one for a larger set of noncentrality parameters.}
	\label{fig: non-centrality-regions}
\end{figure}

\begin{figure}
	\centering
	\includegraphics[width=0.83\linewidth]{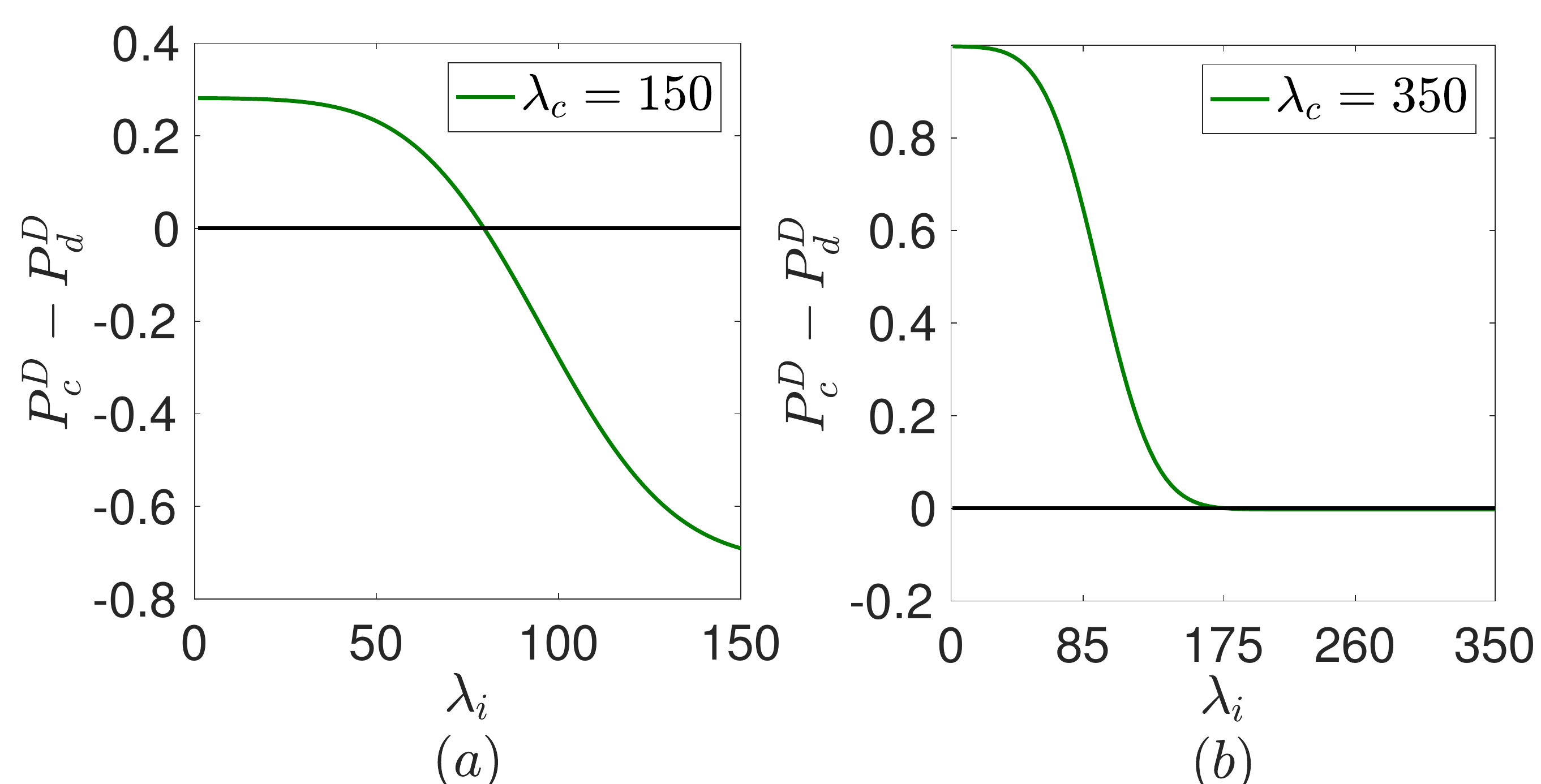}
	\caption{This figure shows the difference of the centralized
          and decentralized detection probabilities as a function of
          $\lambda_i$ for different values of $\lambda_c$. For small
          values of $\lambda_c$, the detection probability of the
          decentralized detector can be substantially larger than its
          centralized counterpart.}
	\label{fig: PDC minus PDD}
\end{figure}

\section{Design of optimal attacks}\label{sec: DOOA}
In this section we consider the problem of designing attacks that
deteriorate the performance of the interconnected system~\eqref{eq:
  subystem without attack} while remaining undetected from the
centralized and decentralized detectors. We measure the degradation
induced by an attack with the expected value of the deviation of the
state trajectory from the origin. We assume that the attack is a
deterministic signal, and thus independent of the noise affecting the
system dynamics and measurements. In particular, for a fixed value of
the probability $P_c^F$ and a threshold $P_c^F \le \delta_c \le 1$, we
consider the optimization problem
\begin{align*}
  \text{{(P.1)}}\quad& \underset{U^a }{\text{max}}
  & & \mathbb{E}\left[ \sum_{k=1}^{T}x(k)^\transpose x(k)\right], \\
                         & \text{subject to}
  & & P^D_c \leq \delta_c, \\
                         &&& x(k+1)=A x(k) + B^a u^a(k) + w(k), 
\end{align*}
where $U^a$ is the deterministic attack input over time horizon $T$ (see \eqref{eq: expanded measurements central}). Notice that, because the attack is deterministic, the objective
function in (P.1) can be simplified by bringing the expectation inside
the summation, and replacing the state equation constraint with the
mean state response. Further, because the system parameters and
$P_c^F$ are fixed, $\tau_c$ and $p_c$ are also fixed, which ensures that $P^D_d$ only depends on noncentrality parameter. This observation along with the fact that $Q(\cdot)$ is increasing function in noncentrality parameter (see Remark \ref{rmk: sys theor interp}) allows us to express the 
detection constraint in terms of $\lambda_c$. Specifically, the optimization problem (P.1)
can be rewritten~as
\begin{align*}
  \text{{(P.2)}}\quad& \underset{U^a}{\text{max}}
  & & \sum_{k=1}^{T}\overline{x}(k)^\transpose\overline{x}(k) \\
                     & \text{subject to}
  & & (U^a)^\transpose  M_c (U^a) \leq \widetilde{\delta}_c,\\
                     &&&\overline x(k+1)=A \overline x(k) + B^a
                         u^a(k) ,
\end{align*}
where we have used that $\mathrm{Cov} \left[x(k)\right]$ is
independent of the attack $u^a(k)$, and 
\begin{align*}
  \mathbb{E}[x(k)^\transpose x(k)] &= \overline{x}(k)^\transpose
                                     \overline{x}(k) + \mathrm{Trace}
                                     \left(\mathrm{Cov} \left[x(k)\right]\right),
\end{align*}
with $\overline{x}(k)=\mathbb{E}[x(k)]$. Further, we have
$\widetilde{\delta}_c=Q^{-1}_{p_c,\tau_c}(\delta_c)$, where
$Q^{-1}_{p_c,\tau_c}(\alpha): [0,1] \to [0,\infty]$ denotes the
inverse of $Q(\tau_c; p_c,\lambda_c)$ for fixed $p_c$ and $\tau_c$,
and $\lambda_c = (U^a)^\transpose M_c (U^a)$, with $M_c$ as in
\eqref{eq: Mi and Mc}. It should be noticed that the attack constraint
in (P.2) essentially limits the (weighted) energy of the attack
signal. We next characterize the solution to the optimization problem
(P.2).


\begin{theorem}{\bf {\it \textbf{(Optimal attack vectors)}}}\label{thm:
    central attack performance} Let $U^*_c$ be any solution of (P.2). Then, there exist a $\gamma_c>0$ such that the pair $(U^*_c, \gamma_c)$ solves the following optimality equations: 
\begin{subequations}
	\begin{align}
	\left[ \mathcal{B}_a^\transpose\mathcal{B}_a-\gamma_c M_c\right]U^*_c+\mathcal{B}_a^\transpose\mathcal{A}x(0)&=0, \label{eq: optimal necessary one}\\
	\left( U^*_c\right)^\transpose M_c (U^*_c)&=\widetilde \delta_c, \label{eq: optimal necessary two} 
	\end{align}
\end{subequations}
where
  \begin{align}\label{eq: mcal A and B}
\mathcal{A} = 
\begin{bmatrix}
A\\
\vdots \\
A^T
\end{bmatrix} \text{ and } \;
\mathcal{B}_a &= 
\begin{bmatrix}
B^a & \cdots  & 0\\
\vdots & \ddots & \vdots \\
A^{T-1}B^a & \ldots & B^a
\end{bmatrix}. 
\end{align}
\end{theorem}

\smallskip 

A proof of Theorem \ref{thm: central attack performance} is postponed
to the Appendix. Theorem \ref{thm: central attack performance} not
only guarantees the existence of optimal attacks, but it also provides
us with necessary conditions to verify if an attack is (locally)
optimal. When the system initial state is zero, we can also quantify
the performance degradation induced by an optimal attack. Let
$\rho_\text{max} (A,B)$ and $\nu_\text{max} (A,B)$ denote a largest
generalized eigenvalue of a matrix pair $(A,B)$ and one of its
associated generalized eigenvectors \cite{DK:05}.

\begin{lemma}{\bf {\it \textbf{(System degradation with zero initial
        state)}}}\label{lma: central attack performance zero initial state} Let
  $x(0)=0$. Then, the optimal solution to (P.2) is
  \begin{align}
    U^*_c&=\left( \sqrt{\frac{\widetilde \delta_c}{(\nu^*)^\transpose M_c(\nu^*)}}\right) \nu^*,
  \end{align}
  and its associated optimal cost is
  \begin{align}\label{eq: optimal cost central}
    J_c^* &=\widetilde \delta_c\,\rho_\text{max}\left( \mathcal{B}_a^\transpose\mathcal{B}_a, M_c\right),
  \end{align}
 where $\nu^*=\nu_\text{max}\left( \mathcal{B}_a^\transpose\mathcal{B}_a, M_c\right)$. 
\end{lemma}

\smallskip 

A proof of Lemma \ref{lma: central attack performance zero initial
  state} is postponed to the Appendix. From \eqref{eq: optimal cost
  central}, notice that the system degradation caused by an optimal
attack depends on the detector's tolerance, as measured by
$\widetilde \delta_c$, and the system dynamics, as measured by
$\rho_\text{max}\left(\cdot\right)$. See Remark \ref{remark:
  degradation large noise} for the influence of processed
measurement's noise uncertainty on the system degradation due to
optimal attacks.

\begin{remark}{\bf {\it \textbf{(Optimal attack vector against
        decentralized detector)}}}\label{remark: optimal attack
    decentralized detector} To characterize the performance
  degradation of the system analytically, we consider a relaxed form
  of detection constraint. Specifically, we design optimal attacks
  subjected to $\overline{P}^D_d\leq \delta_d$ instead of
  $P^D_d\leq \delta_d$, where $\overline{P}^D_d$ is an upper bound on
  $P^D_d$ (see Lemma \ref{lma: ub PDD}). The design of optimal attacks that are
  undetectable from the decentralized detector can be formulated in the following way:
  \begin{align*}
    \text{{(P.3)}}\quad& \underset{U^a}{\text{max}}
    & & \sum_{k=1}^{T}\overline{x}(k)^\transpose\overline{x}(k) \\
                       & \text{subject to}
    & & \sum\limits_{i=1}^{N}(U_i^a)^\transpose  M_i (U_i^a) \leq \widetilde{\delta}_d,\\
                       &&&\overline x(k+1)=A \overline x(k) + B^a
                           u^a(k) ,
  \end{align*}
  where the summation in the detectability constraint follows from
  Lemma \ref{lma: ub PDD} and the fact that
  $\overline{P}^D_d\leq \delta_d$ becomes equivalent to
  $\sum_{i=1}^{N}\lambda_i\leq \widetilde \delta_d$, where $\widetilde{\delta}_d=Q^{-1}_{p_\text{sum},\tau_\text{min}}(\delta_d)$, $p_{\text{sum}}=\sum_{i=1}^Np_i$, and $\tau_{\text{min}}=\min\limits_{1 \leq i \leq N}\tau_i$. Let $\Pi_i$ be a
  permutation matrix such that $U_i^a=\Pi_iU^a$, and let
  $\Pi=\left[\Pi_1^\transpose,\ldots,\Pi_N^\transpose\right]^\transpose$
  and $M_d=\Pi^\transpose\mathrm{blkdiag}(M_1,\ldots,M_N)\Pi$. For
  any solution $U^*_d$ of (P.2), there exist $\gamma_d>0$ such that
  the pair $\left(U^*_d,\gamma_d\right)$ solves the following
  optimality equations:
  	\begin{align*}
  	\left[ \mathcal{B}_a^\transpose\mathcal{B}_a-\gamma_d M_d\right]U^*_d+\mathcal{B}_a^\transpose\mathcal{A}x(0)&=0, \text{ and }\\
  	\left( U^*_d\right)^\transpose M_d (U^*_d)&=\widetilde \delta_d. 
  	\end{align*}
   Further, if $x(0) = 0$, then the largest
  degradation is
  $J^*_d=\widetilde \delta_d\,\rho_\text{max}\left( \mathcal{B}_a^\transpose\mathcal{B}_a, M_d\right) $.\QEDB
\end{remark}

\begin{remark}{\bf {\it \textbf{(Maximum degradation of the system
        performance with respect to system noise)}}}\label{remark:
    degradation large noise} To see the role of noise level, in
  the processed measurements, on the system degradation, we consider
  the following covariance matrices: $\Sigma_{w_i}=\sigma^2I_{n_i}$
  and $\Sigma_{v_i}=\sigma^2I_{r_i}$, for $i \in
  \{1,\ldots,N\}$. Then, from \eqref{eq: optimal cost central} we have
  \begin{align}\label{[eq: noise effect on cost]}
    J^*_c&=\sigma^2\,\widetilde{\delta}_c\,\left[ \rho_\text{max}\left(
           \mathcal{B}_a^\transpose \mathcal{B}_a, \widetilde M_c \right)\right], 
  \end{align}
  where
  $ \widetilde M_c= \left(N_c\mathcal{F}_c^{(a)}\right)^\transpose
  \left[\mathcal{F}_c^{(w)}\left( \mathcal{F}_c^{(w)}\right)
    ^\transpose+I\right]^{-1}\left(N_c\mathcal{F}_c^{(a)}\right)$. From
  \eqref{[eq: noise effect on cost]} we note that the system
  degradation increases with the increase in the noise level, i.e., $\sigma^2$. \QEDB
\end{remark}

\section{Numerical comparison of centralized and decentralized
  detectors}\label{sec: simulations}
In this section, we demonstrate our theoretical findings on the IEEE
RTS-96 power network model \cite{CG-Wo:99}, which we partition into
three subregions as shown in Fig. \ref{fig: rts-network}. We followed the
approach in \cite{fd-fp-fb:12b} to obtain a linear time-invariant
model of the power network, and then discretized it using a sampling
time of $0.01$ seconds. For a 
false alarm probability $P^F_c=P^F_d=0.05$, we consider the family of
attacks
$U^a = \sqrt{ \theta /(\boldsymbol{1}^\transpose M_c \boldsymbol{1})}
\boldsymbol{1}$, where $\boldsymbol{1}$ is the vector of all ones and
$\theta > 0$. It can be shown that the noncentrality parameters
satisfy $\lambda_c=\theta$ and
$\lambda_i=\theta (\boldsymbol{1}^\transpose M_i
\boldsymbol{1})/(\boldsymbol{1}^\transpose M_c \boldsymbol{1})$, and moreover, the choice of vector $\boldsymbol{1}$ is arbitrary and it does not affecting the following results.


\begin{figure}[t]
  \centering
  \includegraphics[width=0.6\linewidth]{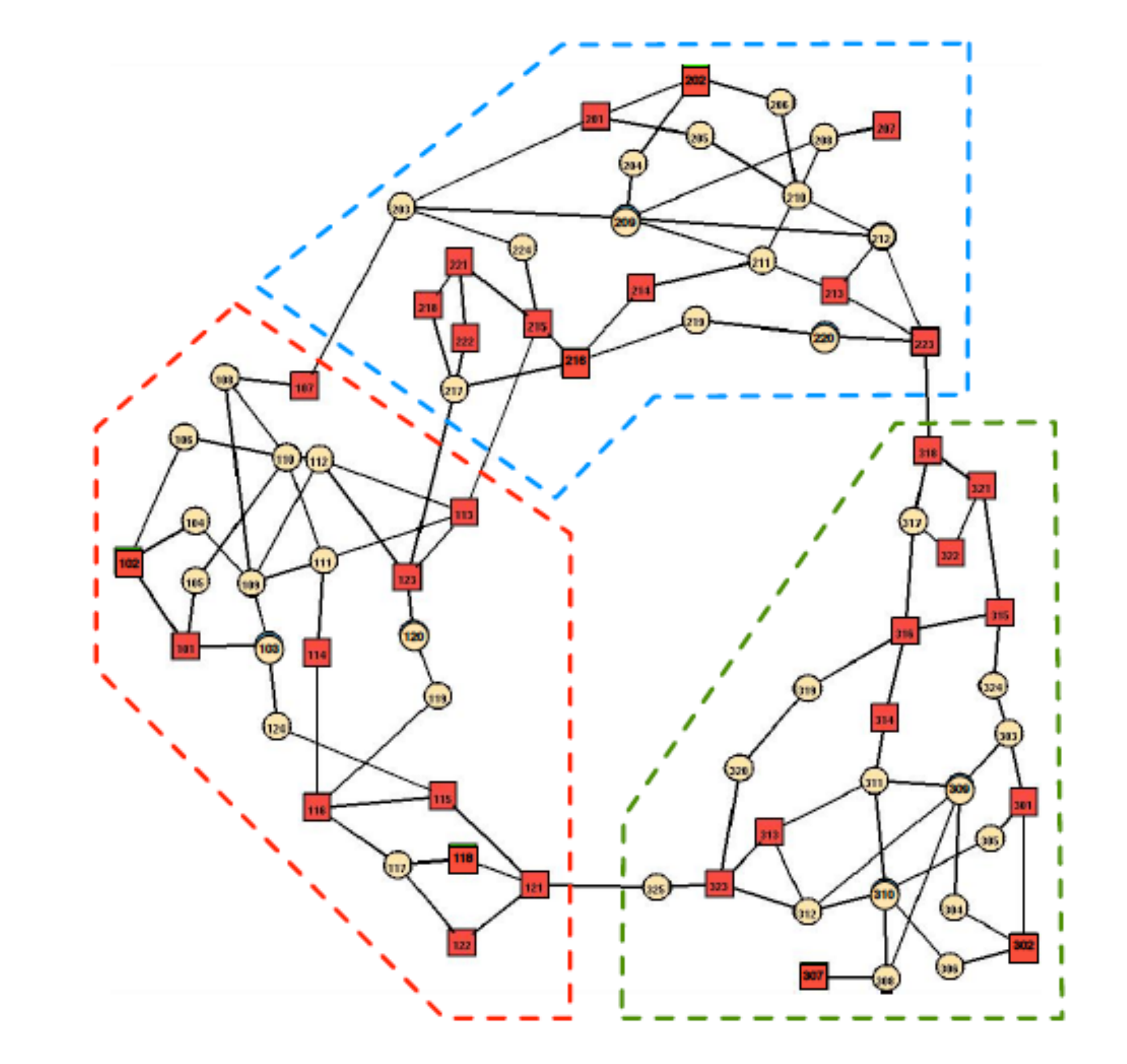}
  \caption {The figure shows a single-line diagram of IEEE RTS96 power
    network, which is composed of three weakly-coupled areas
    (subsystems). The square nodes denote the generators, while the
    circular nodes denotes the load buses of the network
    \cite{fd-fp-fb:12b}.} \label{fig: rts-network}
\end{figure}

\begin{figure}[H]
	\centering
	\includegraphics[width=1.0\linewidth]{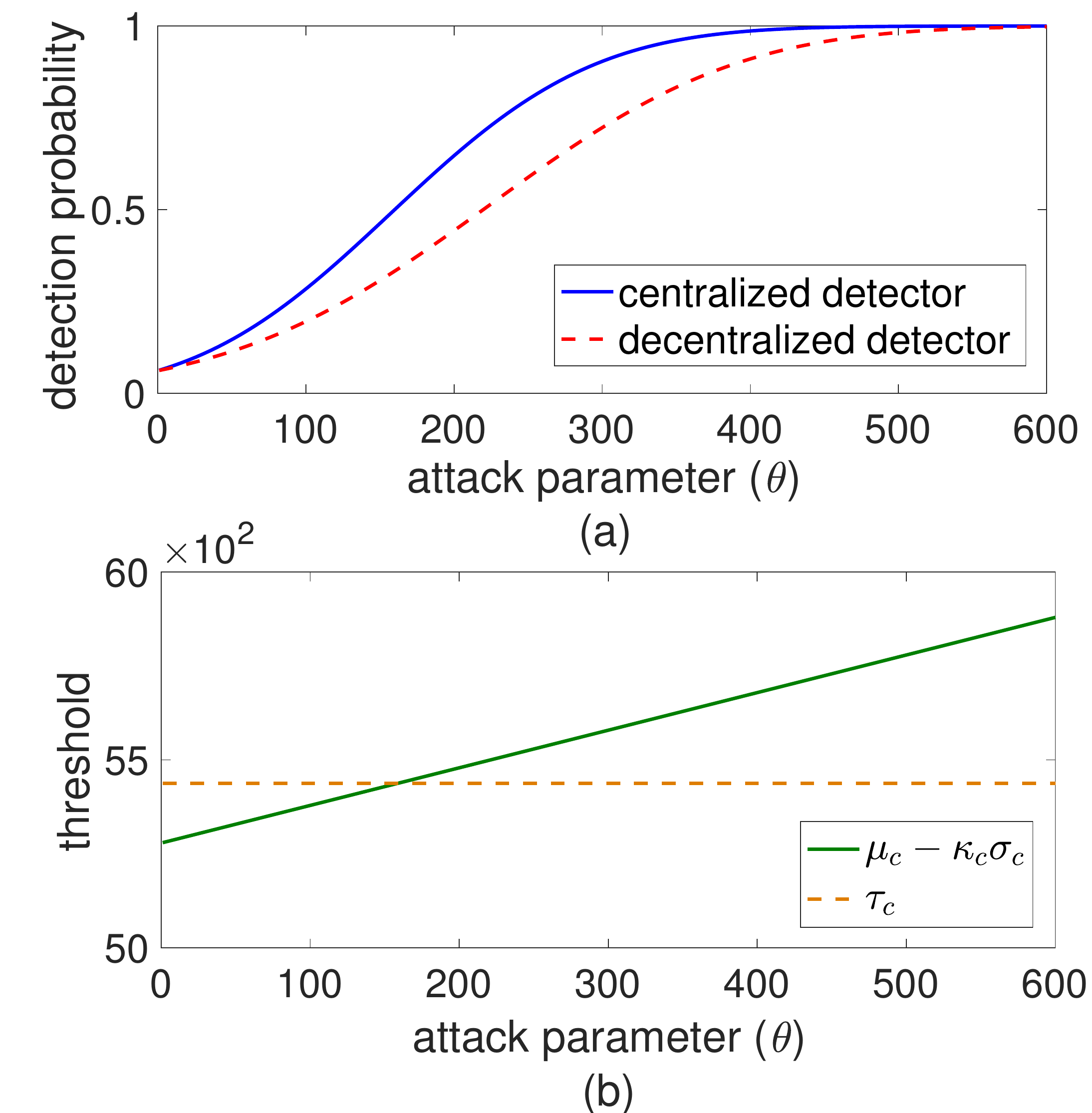}
	\caption {An illustration of a scenario in which the centralized
		detector outperforms the decentralized detector for the 
		IEEE RTS-96 power network. In panel (a), we plot the detection
		probabilities of the detectors with respect to the attack
		parameter $\theta$. Instead, in panel (b) we plot the right (solid line) and
		left hand expressions (dashed line) of the inequality in \eqref{eq: modified
			sufficient condition 1} as a function of $\theta$. For attacks
		such that $\theta > 200$, the sufficient condition \eqref{eq:
			modified sufficient condition 1} holds true, it guarantees that $P^D_c\geq
		P^D_d$.} \label{fig: simulation_fig_1}
\end{figure}

\smallskip 
\noindent\textit{(Illustration of Theorem \ref{thm: Pc greater than
    Pd})} For the measurement horizon of $T=100$ seconds, the values of $p_c$ and $\tau_c$ are $5130$ and $5480.6$, respectively. Fig. \ref{fig: simulation_fig_1} show that the detection
probabilities of the centralized and decentralized detectors increase
monotonically with the attack parameter $\theta$. As predicted by
the sufficient condition \eqref{eq: modified sufficient condition 1}
and shown in Fig.~\ref{fig: simulation_fig_1}, the centralized detector is guaranteed to
outperform the decentralized detector when $\theta > 173$. This
figure also shows that our condition is conservative, because $P^D_c
\geq P^D_d$ for all values of $\theta$ as shown in
Fig.~\ref{fig: simulation_fig_1}.

\begin{figure}[H]
	\centering
	\includegraphics[width=1.0\linewidth]{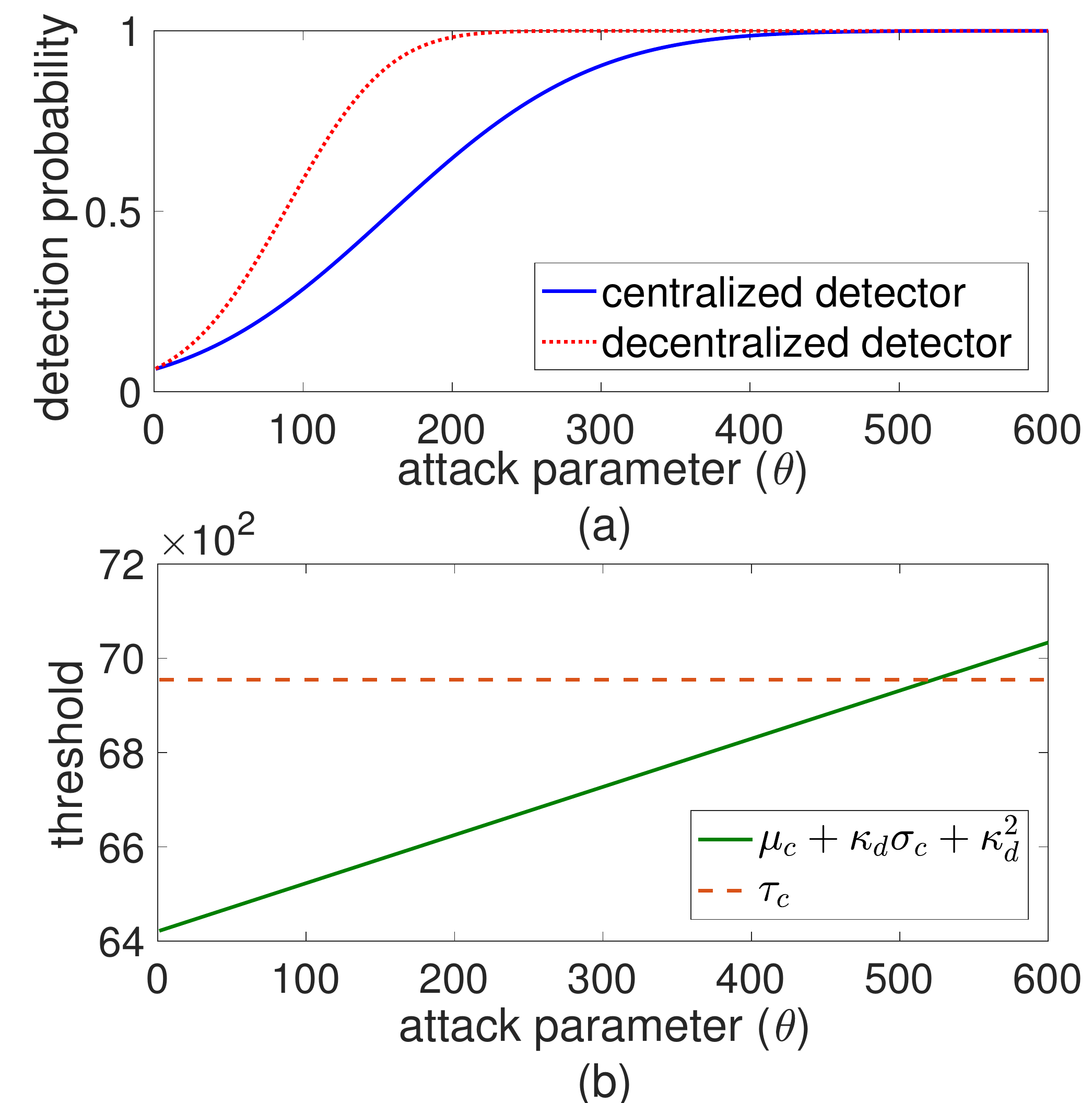}
	\caption {An illustration of a scenario in which the decentralized
		detector outperforms the centralized detector for the 
		IEEE RTS-96 power network. In panel (a), we plot the detection
		probabilities of the detectors with respect to the attack
		parameter $\theta$. Instead, in panel (b) we plot the right (solid line) and
		left hand expressions (dashed line) of the inequality in  \eqref{eq: modified sufficient condition 2} as a function of
		$\theta$. For attacks such that $\theta < 500$, the sufficient
		condition \eqref{eq: modified sufficient condition 2} holds true,
		and it guarantees that
		$P^D_c\leq P^D_d$.} \label{fig: simulation_fig_2}
\end{figure}

\begin{figure}[H]
	\centering
	\includegraphics[width=1.0\linewidth]{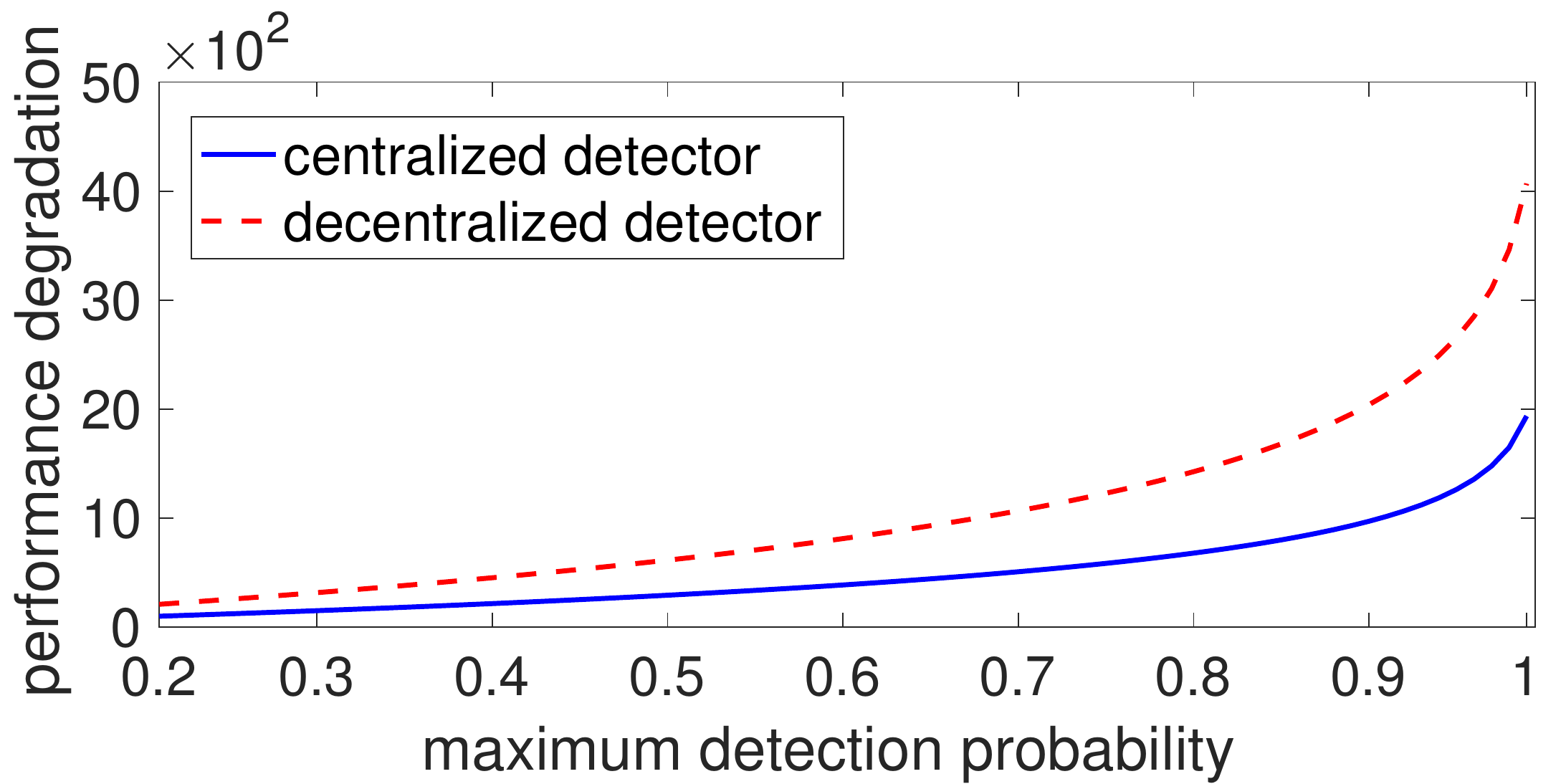}
	\caption {This figure shows the performance degradation induced by
		undetectable optimal attacks on the IEEE RTS-96 power network. The performance
		degradation is computed using the optimal cost $J^*_c$ and $J^*_d$
		derived in Lemma \ref{lma: central attack performance zero initial
			state} and Remark \ref{remark: optimal attack decentralized
			detector}, respectively. Instead, the maximum detection
		probability is given by the tuning parameters
		$\delta_c$ and $\delta_d$ in the detection
		probability constraints of the optimization problems (P.2) and
		(P.3), respectively.} \label{fig: simulation_fig_3}
\end{figure}

\smallskip 
\noindent\textit{(Illustration of Theorem \ref{thm: Pd greater than
		Pc})} Contrary to the previous example, by letting $T=125$ seconds, we obtain $p_c=6755$ and $\tau_c=6947.3$. For these choice of parameters, the decentralized detector is guaranteed to
outperform the centralized detector when $\theta\leq 511$. This
behavior is predicted by our sufficient condition \eqref{eq: modified
	sufficient condition 2}, and it is illustrated in Fig.~\ref{fig: simulation_fig_2}. As in
the previous example, the estimation provided by our condition
\eqref{eq: modified sufficient condition 2} is conservative, as
illustrated in Fig.~\ref{fig: simulation_fig_2}.

\smallskip 
\noindent\textit{(Illustration of Lemma \ref{lma: central attack
		performance zero initial state})} In Fig. \ref{fig:
	simulation_fig_3} we compare the performance degradation induced by
the optimal attacks designed according to the optimization problems (P.2) and (P.3) with zero initial conditions. In
particular, we plot the optimal costs $J^*_c$ and $J^*_d$ against the
tolerance levels $\widetilde \delta_c$ and $\widetilde \delta_d$,
respectively. As expected, the performance degradation is proportional
to the tolerance levels and, for the considered setup, it is larger
in the case of the decentralized detector.

\section{Conclusion}\label{sec: conclusions}
In this work we compare the performance of centralized and
decentralized schemes for the detection of attacks in
stochastic interconnected systems. In addition to quantifying the
performance of each detection scheme, we prove the counterintuitive
result that the decentralized scheme can, at times, outperform its
centralized counterpart, and that this behavior results due to the simple
versus composite nature of the attack detection problem. We illustrate
our findings through academic examples and a case study based on the
IEEE RTS-96 power system. Several questions remain of interest for
future investigation, including the characterization of optimal
detection schemes, an analytical comparison of the degradation induced
by undetectable attacks as a function of the detection scheme, and the
analysis of iterative detection strategies.


\bibliographystyle{IEEEtran}
\bibliography{./BIB}

\begin{thebibliography}{10}
\providecommand{\url}[1]{#1}
\csname url@samestyle\endcsname
\providecommand{\newblock}{\relax}
\providecommand{\bibinfo}[2]{#2}
\providecommand{\BIBentrySTDinterwordspacing}{\spaceskip=0pt\relax}
\providecommand{\BIBentryALTinterwordstretchfactor}{4}
\providecommand{\BIBentryALTinterwordspacing}{\spaceskip=\fontdimen2\font plus
\BIBentryALTinterwordstretchfactor\fontdimen3\font minus
  \fontdimen4\font\relax}
\providecommand{\BIBforeignlanguage}[2]{{%
\expandafter\ifx\csname l@#1\endcsname\relax
\typeout{** WARNING: IEEEtran.bst: No hyphenation pattern has been}%
\typeout{** loaded for the language `#1'. Using the pattern for}%
\typeout{** the default language instead.}%
\else
\language=\csname l@#1\endcsname
\fi
#2}}
\providecommand{\BIBdecl}{\relax}
\BIBdecl

\bibitem{fp-fd-fb:10y}
F.~Pasqualetti, F.~Dörfler, and F.~Bullo, ``Attack detection and identification
  in cyber-physical systems,'' \emph{IEEE Transactions on Automatic Control},
  vol.~58, no.~11, pp. 2715--2729, Nov 2013.

\bibitem{YL-AD:16}
Y.~Lun, A.~D'Innocenzo, I.~Malavolta, M.~Domenica, and D.~Benedetto,
  ``Cyber-physical systems security: a systematic mapping study,''
  \emph{arxiv}, 2016, available at https://arxiv.org/pdf/1605.09641.pdf.

\bibitem{JC-RJ:99}
J.~Chen and R.~Patton, \emph{Robust Model-Based Fault Diagnosis for Dynamic
  Systems}.\hskip 1em plus 0.5em minus 0.4em\relax Springer-Verlag New York,
  1999.

\bibitem{YY-QZ-FS-QW-TB}
Y.~Yuan, Q.~Zhu, F.~Sun, Q.~Wang, and T.~Basar, ``Resilient control of
  cyber-physical systems against denial-of-service attacks,'' in \emph{2013 6th
  International Symposium on Resilient Control Systems (ISRCS)}, Aug 2013, pp.
  54--59.

\bibitem{HF-PT-SD}
H.~Fawzi, P.~Tabuada, and S.~Diggavi, ``Secure estimation and control for
  cyber-physical systems under adversarial attacks,'' \emph{IEEE Transactions
  on Automatic Control}, vol.~59, no.~6, pp. 1454--1467, 2014.

\bibitem{SS-MG}
S.~Sridhar and M.~Govindarasu, ``Model-based attack detection and mitigation
  for automatic generation control,'' \emph{IEEE Transactions on Smart Grid},
  vol.~5, no.~2, pp. 580--591, March 2014.

\bibitem{LL-ME-QD-VA-ZH}
L.~Liu, M.~Esmalifalak, Q.~Ding, V.~A. Emesih, and Z.~Han, ``Detecting false
  data injection attacks on power grid by sparse optimization,'' \emph{IEEE
  Transactions on Smart Grid}, vol.~5, no.~2, pp. 612--621, March 2014.

\bibitem{HZ-PC-LS-JC}
H.~Zhang, P.~Cheng, L.~Shi, and J.~Chen, ``Optimal denial-of-service attack
  scheduling with energy constraint,'' \emph{IEEE Transactions on Automatic
  Control}, vol.~60, no.~11, pp. 3023--3028, Nov 2015.

\bibitem{YM-BS:16}
Y.~Mo and B.~Sinopoli, ``On the performance degradation of cyber-physical
  systems under stealthy integrity attacks,'' \emph{IEEE Transactions on
  Automatic Control}, vol.~61, no.~9, pp. 2618--2624, Sept 2016.

\bibitem{FM-QZ-MP-GJ}
F.~Miao, Q.~Zhu, M.~Pajic, and G.~J. Pappas, ``Coding schemes for securing
  cyber-physical systems against stealthy data injection attacks,'' \emph{IEEE
  Transactions on Control of Network Systems}, vol.~4, no.~1, pp. 106--117,
  March 2017.

\bibitem{YC-SK-JM}
Y.~Chen, S.~Kar, and J.~M.~F. Moura, ``Dynamic attack detection in
  cyber-physical systems with side initial state information,'' \emph{IEEE
  Transactions on Automatic Control}, vol.~62, no.~9, pp. 4618--4624, Sept
  2017.

\bibitem{CK-IH:18}
C.~Kwon and I.~Hwang, ``Reachability analysis for safety assurance of
  cyber-physical systems against cyber attacks,'' \emph{IEEE Transactions on
  Automatic Control}, vol.~63, no.~7, pp. 2272--2279, July 2018.

\bibitem{fd-fp-fb:11t}
F.~D{\"o}rfler, F.~Pasqualetti, and F.~Bullo, ``Distributed detection of
  cyber-physical attacks in power networks: {A} waveform relaxation approach,''
  in \emph{Allerton Conf.\ on Communications, Control and Computing}, Allerton,
  IL, USA, Sep. 2011, pp. 1486--1491.

\bibitem{NI-IS:14}
H.~Nishino and H.~Ishii, ``Distributed detection of cyber attacks and faults
  for power systems,'' \emph{IFAC Proceedings Volumes}, vol.~47, no.~3, pp.
  1932 -- 11\,937, 2014.

\bibitem{MN-YY-XW:18}
M.~N. Kurt, Y.~Y{\i}lmaz, and X.~Wang, ``Distributed quickest detection of
  cyber-attacks in smart grid,'' \emph{IEEE Transactions on Information
  Forensics and Security}, vol.~13, no.~8, pp. 2015--2030, Aug 2018.

\bibitem{JZ-LM:18}
J.~Zhao and L.~Mili, ``Power system robust decentralized dynamic state
  estimation based on multiple hypothesis testing,'' \emph{IEEE Transactions on
  Power Systems}, vol.~33, no.~4, pp. 4553--4562, July 2018.

\bibitem{EM-AK-DK:15}
E.~M. Hammad, A.~K. Farraj, and D.~Kundur, ``A resilient feedback linearization
  control scheme for smart grids under cyber-physical disturbances,'' in
  \emph{2015 IEEE Power Energy Society Innovative Smart Grid Technologies
  Conference (ISGT)}, Feb 2015, pp. 1--5.

\bibitem{RRT-NRS:81}
R.~R. Tenney and N.~R. Sandell, ``Detection with distributed sensors,''
  \emph{IEEE Transactions on Aerospace and Electronic Systems}, vol.~17, no.~4,
  pp. 501--510, July 1981.

\bibitem{PK:97}
P.~Varshney, \emph{Distributed Detection and Data Fusion}.\hskip 1em plus 0.5em
  minus 0.4em\relax Springer-Verlag New York, 1997.

\bibitem{JC-VV;03}
J.~Chamberland and V.~V. Veeravalli, ``Decentralized detection in sensor
  networks,'' \emph{IEEE Transactions on Signal Processing}, vol.~51, no.~2,
  pp. 407--416, Feb 2003.

\bibitem{SA-VV-DL:05}
S.~Appadwedula, V.~V. Veeravalli, and D.~L. Jones, ``Energy-efficient detection
  in sensor networks,'' \emph{IEEE Journal on Selected Areas in
  Communications}, vol.~23, no.~4, pp. 693--702, April 2005.

\bibitem{vr-mm-cg:15}
V.~Reppa, M.~M. Polycarpou, and C.~G. Panayiotou, ``Distributed sensor fault
  diagnosis for a network of interconnected cyberphysical systems,'' \emph{IEEE
  Transactions on Control of Network Systems}, vol.~2, no.~1, pp. 11--23, March
  2015.

\bibitem{XG-CE:08}
X.-G. Yan and C.~Edwards, ``Robust decentralized actuator fault detection and
  estimation for large-scale systems using a sliding mode observer,''
  \emph{Int. Journal of Control}, vol.~81, no.~4, pp. 591--606, 2008.

\bibitem{IS-AM-HS-KH:11}
I.~Shames, A.~M.~H. Teixeira, H.~Sandberg, and K.~H. Johansson, ``Distributed
  fault detection for interconnected second-order systems,'' \emph{Automatica},
  vol.~47, no.~12, pp. 2757 -- 2764, 2011.

\bibitem{RM-TP-MM:12}
R.~M.~G. Ferrari, T.~Parisini, and M.~M. Polycarpou, ``Distributed fault
  detection and isolation of large-scale discrete-time nonlinear systems: An
  adaptive approximation approach,'' \emph{IEEE Transactions on Automatic
  Control}, vol.~57, no.~2, Feb 2012.

\bibitem{RA-VK-FP:18}
A.~Rajasekhar, V.~Katewa, and F.~Pasqualetti, ``Attack detection in stochastic
  interconnected systems: Centralized vs decentralized detectors,'' in
  \emph{57th IEEE Conference on Decision and Control (CDC)}, 2018, {T}o appear.

\bibitem{HVP:94}
H.~V. Poor, \emph{An Introduction to Signal Detection and Estimation}.\hskip
  1em plus 0.5em minus 0.4em\relax Springer-Verlag New York, 1994.

\bibitem{DK:05}
K.~Daniel, \emph{Numerical Methods for General and Structure Eigenvalue
  Problems}.\hskip 1em plus 0.5em minus 0.4em\relax Springer-Verlag Berlin
  Heidelberg, 2005.

\bibitem{CG-Wo:99}
C.~Grigg, P.~Wong, P.~Albrecht, R.~Allan, M.~Bhavaraju, R.~Billinton, Q.~Chen,
  C.~Fong, S.~Haddad, S.~Kuruganty, W.~Li, R.~Mukerji, D.~Patton, N.~Rau,
  D.~Reppen, A.~Schneider, M.~Shahidehpour, and C.~Singh, ``The {IEEE}
  reliability test system-1996. a report prepared by the reliability test
  system task force of the application of probability methods subcommittee,''
  \emph{IEEE Transactions on Power Systems}, vol.~14, no.~3, pp. 1010--1020,
  Aug 1999.

\bibitem{fd-fp-fb:12b}
F.~D{\"o}rfler, F.~Pasqualetti, and F.~Bullo, ``Continuous-time distributed
  observers with discrete communication,'' \emph{IEEE Journal of Selected
  Topics in Signal Processing}, vol.~7, no.~2, pp. 296--304, 2013.

\bibitem{anderson}
T.~Anderson, \emph{An Introduction to Multivariate Statistical Analysis}.\hskip
  1em plus 0.5em minus 0.4em\relax Wiley, New York, 1958.

\bibitem{EC-SZ:13}
E.~K. Chong and S.~H. Zak., \emph{An Introduction to Optimization}.\hskip 1em
  plus 0.5em minus 0.4em\relax Wiley, New York, 2013.

\bibitem{JNL-KS-BN:95}
N.~L. Johnson, S.~Kotz, and N.~Balakrishnan, \emph{Continuous univariate
  distributions, Volume 2}.\hskip 1em plus 0.5em minus 0.4em\relax Wiley \&
  Sons, 1995.

\bibitem{birge}
L.~Birgé, ``An alternative point of view on {L}epski's method,''
  \emph{Institute of Mathematical Statistics}, vol.~36, pp. 113--133, 2001.

\end{thebibliography}

\setcounter{section}{1} 
\renewcommand{\thetheorem}{A.\arabic{theorem}}
\section*{APPENDIX}  


\noindent{\it Proof of Lemma \ref{thm: stat prop of processed}}:

\smallskip 
Since the attack vectors $U^a_i$ and $U^a$ are deterministic, and $W_i$, $V_i$, $V$, and $W$ are zero mean random vectors, from the linearity of the expectation operator it follows from \eqref{eq: kernels} that 
\begin{align*}
	\beta_i&\triangleq \mathbb{E}[\widetilde Y_i]=N_i\mathcal{F}^{(a)}_iU^a_i, \text{ and }\\
	\beta_c&\triangleq \mathbb{E}[\widetilde Y_c]=N_c\mathcal{F}^{(a)}_cU^a_c.
\end{align*}
Further, from the properties of $\mathrm{Cov}[\cdot]$, we have the following: 
	\begin{align*}
		\Sigma_i&\triangleq \mathrm{Cov}\left[ \widetilde Y_i\right]\\
		&=N_i\mathrm{Cov}\left[ Y_i\right]N_i^\transpose\\
		&\overset{(a)}{=}N_i\left[ \mathrm{Cov}\left[\mathcal{F}_i^{(w)}W_i\right]+\mathrm{Cov}[V_i] \right]N_i^\transpose\\
		&\overset{(b)}{=}N_i\left[\left( \mathcal{F}_i^{(w)}\right) \mathrm{Cov}\left[W_i\right]\left( \mathcal{F}_i^{(w)}\right)^\transpose +\mathrm{Cov}[V_i] \right]N_i^\transpose\\
		&=N_i\left[\left( \mathcal{F}_i^{(w)}\right) \left(I_T\otimes \Sigma_{w_i}\right)\left( \mathcal{F}_i^{(w)}\right)^\transpose+\left( I_T\otimes \Sigma_{v_i}\right)  \right]N_i^\transpose,
	\end{align*}
	where (a) follows because the measurement and process noises are independent of each other. Instead, (b) is due to the fact that the noise vectors are independent and identically distributed. Similar analysis also results in the expression of $\Sigma_c$, and hence the details are omitted. Finally, by invoking the fact that linear transformations preserve Gaussianity, the distribution of $\widetilde Y_i$ and $\widetilde Y_c$ is Gaussian as well. \QED

\medskip 
\noindent{\it Proof of Lemma \ref{lma: PF and PD via Q-functions}}:

\smallskip 
From the statistics and distributional form of $\widetilde Y_i$ and $\widetilde Y_c$ (see \eqref{eq: mean and covariance}), and threshold tests defined in \eqref{eq: test_stat_local} and \eqref{eq: test_stat_central}, it follows from \cite[Theorem 3.3.3]{anderson} that, under 
\begin{enumerate}
	\item null hypothesis $H_0$: $\Lambda_i\sim \chi^2(p_i)$ and $\Lambda_c\sim \chi^2(p_c)$, where $p_i$ and $p_c$ are defined in \eqref{eq: lambda 1}.
	\item alternative  hypothesis $H_1$: $\Lambda_i\sim \chi^2(p_i,\lambda_i)$ and $\Lambda_c\sim \chi^2(p_c,\lambda_c)$, where $\lambda_i=\beta_i^\transpose\Sigma_i^{-1}\beta_i$ and $\lambda_c=\beta_c^\transpose\Sigma_c^{-1}\beta_c$. 
\end{enumerate}
By substituting $\beta_i=N_i\mathcal{F}^{(a)}_iU^a_i$ and $\beta_c=N_c\mathcal{F}^{(a)}_cU^a_c$ (see Lemma \ref{thm: stat prop of processed}) and rearranging the terms, we get the expressions of $\lambda_i$ and $\lambda_c$ in \eqref{eq: lambda 1}. Finally, from the aforementioned distributional forms of $\Lambda_i$ and $\Lambda_c$, it now follows that the false alarm and the detection probabilities of the tests \eqref{eq: test_stat_local} and \eqref{eq: test_stat_central} are the right tail probabilities (represented by $Q(\cdot)$ function) of the central and noncentral chi-squared distributions, respectively. Hence, the expressions in \eqref{eq: test probabilities} follow. \QED

\medskip 
\noindent{\it Proof of Lemma \ref{lma: parameters}}: 

\smallskip 
Without loss of generality let $i=1$. Thus, it suffices to show that a) $ p_1\leq p_c$ and b) $\lambda_1\leq \lambda_c$. 

\smallskip 
\noindent {\bf Case (a)}: For brevity, define
\begin{align}\label{eq: sigma tildas}
\begin{split}
\widetilde \Sigma_i&= \left[ \left( \mathcal{F}^{(w)}_i\right)\left(I_T\otimes \Sigma_{w_i}\right) \left( \mathcal{F}^{(w)}_i\right)^\transpose+\left(I_T\otimes \Sigma_{v_i}\right)\right] \text{ and }\\
\widetilde \Sigma_c&= \left[\left( \mathcal{F}^{(w)}_c\right) \left(I_T\otimes
\Sigma_{w}\right) \left(
\mathcal{F}^{(w)}_c\right)^\transpose+\left(I_T\otimes
\Sigma_{v}\right)\right],  
\end{split}
\end{align}
and note that $\widetilde \Sigma_i>0$ and $\widetilde \Sigma_c>0$. From Lemma \ref{thm: stat prop of processed}, Lemma \ref{lma: PF and PD via Q-functions}, and \eqref{eq: sigma tildas}, we have 
\begin{align*}
p_c=\text{Rank}(\Sigma_c)&=\text{Rank}\left( \left( N_c\widetilde \Sigma_c^{1/2}\right)  \left( N_c\widetilde \Sigma_c^{1/2}\right)^\transpose\right) \\
&=\text{Rank}\left( N_c\widetilde \Sigma^{1/2}\right)=\text{Rank}\left( N_c\right).
\end{align*}
Similarly, $p_1=\text{Rank}\left( N_1\right)$. Since, $N_1^\transpose$ and $N_c^\transpose$ are a basis vectors of the null spaces $\mathcal{N}_1^L$ and $\mathcal{N}_c^L$ (see \eqref{eq: null spaces}) respectively, it follows from Proposition \ref{prop: null space containment} that $p_1\leq p_c$. 

\smallskip 
\noindent {\bf Case (b)}: As the proof for this result is rather long and tedious, we break it down in to multiple steps:
\begin{itemize}
	\item \textit {Step 1}: Express $\lambda_1$ and $\lambda_c$ using the statistics of a permuted version of $Y_c$. 
	\item \textit {Step 2}: Obtain lower bound on $\lambda_c$, which depends on the statistics of the measurements pertaining to Subsystem 1. 
	\item \textit {Step 3}: Show that $\lambda_1$ is less than bound in Step $2$. 
\end{itemize}

\noindent \emph{\textit{Step 1 (alternative form of $\lambda_1$ and $\lambda_c$)}:}

Notice that $\lambda_1$ and $\lambda_c$ in \eqref{eq: lambda 1} can be expressed as $\lambda_1=\beta_1^\transpose\Sigma_1^{-1}\beta_1$ and $\lambda_c=\beta_c^\transpose\Sigma_c^{-1}\beta_c$, respectively, where $\beta_1$, $\beta_c$, $\Sigma_1$, and $\Sigma_c$ are defined in Lemma \ref{thm: stat prop of processed}. For convenience, we express $\lambda_1$ and $\lambda_c$ in an alternative way. Let $i \in \{1,\ldots,N\}$ and consider the $i-$th sensor measurements of \eqref{eq: overall with attack}
\begin{align}\label{eq: yci}
	y_{c,i}(k)= \underbrace{\begin{bmatrix}
	0 & \cdots & C_i & \cdots & 0
	\end{bmatrix}}_{\triangleq C_{c,i}}x(k)+v_i(k).
\end{align}
Also, define $Y_{c,i}=\begin{bmatrix}
y_{c,i}^\transpose(1) & \ldots & y_{c,i}^\transpose(T)
\end{bmatrix}^\transpose$ and $\widehat Y_c=\begin{bmatrix}
Y_{c,1}^\transpose & \ldots & Y^\transpose_{c,N}
\end{bmatrix}$.
Now, from \eqref{eq: yci} and state equation in \eqref{eq: overall with attack}, $Y_{c,i}$ can be expanded as
\begin{align*}
Y_{c,i}=\mathcal{O}_{c,i}x(0)+\mathcal{F}^{(a)}_{c,i}U^a+\mathcal{F}^{(w)}_{c,i}W+V_i,  
\end{align*}
where the matrices $\mathcal{O}_{c,i}$, $\mathcal{F}^{(a)}_{c,i}$, and $\mathcal{F}^{(w)}_{c,i}$ are similar to the matrices defined in Section II-A. By substituting the above decomposition of $Y_{c,i}$ in $\widehat Y$ we have 
\begin{align*}
\widehat Y_c&=\underbrace{\begin{bmatrix}
\mathcal{O}_{c,1}\\
\vdots\\
\mathcal{O}_{c,N}
\end{bmatrix}}_{\widehat{\mathcal{O}}_c}x(0)+\underbrace{\begin{bmatrix}
\mathcal{F}^{(a)}_{c,1}\\
\vdots\\
\mathcal{F}^{(a)}_{c,N}
\end{bmatrix}}_{\widehat{\mathcal{F}}^{a}_c}U^a+\underbrace{\begin{bmatrix}
\mathcal{F}^{(w)}_{c,1}\\
\vdots\\
\mathcal{F}^{(w)}_{c,N}
\end{bmatrix}}_{\widehat{\mathcal{F}}^{w}_c}W+V.
\end{align*}
 Moreover, from the distributional assumptions on $W$ and $V$, it readily follows that (similarly to the proof of Lemma \ref{thm: stat prop of processed}),
\begin{align}\label{eq: y hat distribution}
\widehat Y_c \sim  \mathcal{N}\left(\widehat{\mathcal{O}}_cx(0)+\widehat{\mathcal{F}}^{a}_cU^a, \Sigma\right),
\end{align}
where $\Sigma=\left( \widehat{\mathcal{F}}^{w}_c \right) \left( I_T\otimes\Sigma_w\right)  \left(\widehat{\mathcal{F}}^{w}_c \right)^\transpose+\left( I_T\otimes \Sigma_v\right) $, and $\Sigma_w$ and $\Sigma_v$ are defined same as in Lemma \ref{thm: stat prop of processed}. 

Now, consider the measurement equation $y_i(k)$ in 
\eqref{eq: subystem without attack} and note that $C_{c,i}x(k)=C_ix_i(k)$. Thus, $y_i(k)=y_{c,i}(k)$, for all $i \in \{1,\ldots,N\}$ and $k \in \mathbb{N}$. From this observation it follows that $Y_i=Y_{c,i}=\Pi_i \widehat Y_c$, where $\Pi_i$ is a selection matrix. Let $\widetilde{N}_i=N_i\Pi_i$ and note that $\widetilde{N}_i \widehat{\mathcal{O}}=N_i\mathcal{O}_{c,i}$. Further from Proposition \ref{prop: null space containment} we have $N_i\mathcal{O}_{c,i}=0$. With these facts in place, from Lemma \ref{thm: stat prop of processed} we now have
\begin{align}\label{eq: new beta_i}
\begin{split}
	\beta_i&=\widetilde{N}_i\widehat{\mathcal{F}}^{a}_cU^a, \text{ and }\\
\Sigma_i&=\widetilde{N}_i\Sigma\widetilde{N}_i^\transpose. 
\end{split}
\end{align}
 Similarly, since $\widehat Y_c$ is just a rearrangement of $Y_c$ (see \eqref{eq: agree central}),
 there exists a permutation matrix $Q$ such that $Y_c=Q\widehat Y_c$, and, ultimately $\widetilde Y_c=N_cY_c=N_cQ\widehat Y_c$. Thus,
\begin{align}\label{eq: beta c new}
\begin{split}
\beta_c&=N_cQ\widehat{\mathcal{F}}^{a}_cU^a, \text{ and }\\
  \Sigma_c&= N_cQ\Sigma(N_cQ)^\transpose. 
\end{split}
\end{align}
Let $z=\widehat{\mathcal{F}}^{a}_cU^a$. From \eqref{eq: new beta_i} and \eqref{eq: beta c new} we have
\begin{align}\label{eq: lambda new}
\begin{split}
\lambda_1&=z^\transpose\widetilde N_1^\transpose\left[\widetilde N_1\Sigma \widetilde N_1^\transpose \right]^{-1}  \widetilde N_1,\\
\lambda_c&=z^\transpose \left( N_cQ\right) ^\transpose\left[\left( N_cQ\right) \Sigma  \left( N_cQ\right) ^\transpose \right]^{-1} \left( N_cQ\right). \\
\end{split}
\end{align}

\smallskip 
\noindent \emph{\textit{Step 2 (lower bound on $\lambda_c$)}:} 

Since, $Y_c=N_cY_c=N_cQ\widehat Y_c$, it follows that $N_cQ$ is the basis of the null space $\widehat{\mathcal{O}}_c$. Further, the row vectors of $\mathcal{O}_{c,i}$ and $\mathcal{O}_{c,j}$ are linearly independent, whenever $i \ne j$. Using these facts we can define $N_{c,i}=\begin{bmatrix}
N_{c,i}^i & \cdots & N_{c,i}^N
\end{bmatrix}$ such that $N_cQ=\begin{bmatrix}
N_{c,1}^\transpose&
\cdots&
N_{c,N}^\transpose
\end{bmatrix}^\transpose$, 
where $N_{c,i}^i\mathcal{O}_{c,i}=0$. Let $P_1=\begin{bmatrix}
\left(N_{c,2}\right)^\transpose & \cdots &  \left(N_{c,N}\right)^\transpose
\end{bmatrix}^\transpose$ and note that 
\begin{align*}
\left( N_cQ\right) \Sigma\left( N_cQ\right) ^\transpose&=\begin{bmatrix}
N_{c,1}\\
P_1
\end{bmatrix}\Sigma
\begin{bmatrix}
N_{c,1}^\transpose & P_1^\transpose
\end{bmatrix}\\
&=\begin{bmatrix}
N_{c,1}\Sigma N_{c,1}^\transpose & N_{c,1}\Sigma P_1^\transpose\\
N_{c,1}^\transpose\Sigma P_1 & P_1^\transpose\Sigma P_1
\end{bmatrix}.
\end{align*}
Let $S_1=N_{c,1}\Sigma N_{c,1}^\transpose$. Since $\Sigma >0$, it follows that both the matrices $S_1$ and $P_1^\transpose\Sigma P_1$ are invertible. Hence, from Schur's complement, there exists a matrix $X\geq 0$ such that 
\begin{align}\label{eq: schur one}
\left[ \left( N_cQ\right) \Sigma\left( N_cQ\right) ^\transpose\right]^{-1}=\begin{bmatrix}
S_1^{-1}  & 0 \\
0 & 0
\end{bmatrix}+X. 
\end{align}
Similarly, consider the following partition of $\Sigma$:
\begin{align*}
	\Sigma=\begin{bmatrix}
		\Sigma_{11} & \Sigma_{12}\\
		\Sigma_{21} & \Sigma_{22}
	\end{bmatrix},
\end{align*}
where $\Sigma_{11}>0$ and $\Sigma_{22}>0$, and let $S_2=(N^1_{c,1})\Sigma_{11}(N^1_{c,1})^\transpose$. Invoking Schur's complement, we have the following:
\begin{align}\label{eq: schur two}
S_1^{-1}=\begin{bmatrix}
S_2 ^{-1}  & 0 \\
0 & 0
\end{bmatrix}+Y, 
\end{align}
where $Y\geq 0$. Substituting\eqref{eq: schur one} and \eqref{eq: schur two} in \eqref{eq: lambda new}, we have 
\begin{align}\label{eq: lambda_c_inequality}
\lambda_c&=
z^\transpose (N_cQ)^\transpose\begin{bmatrix}
S_1^{-1}  & 0 \\
0 & 0
\end{bmatrix}(N_cQ)z+\underbrace{z^\transpose(N_cQ)^\transpose X(N_cQ)z}_{\geq 0}\nonumber\\
&\geq\left[ 
(N_{c,1}z)^\transpose\,
(P_1z)^\transpose
\right] \begin{bmatrix}
S_1^{-1}  & 0 \\
0 & 0
\end{bmatrix}\begin{bmatrix}
N_{c,1}z\\
P_1z
\end{bmatrix}\nonumber\\
&=z^\transpose \left( N_{c,1}^\transpose S_1^{-1}N_{c,1}\right) z\nonumber \\
&=z^\transpose (N_{c,1})^\transpose\begin{bmatrix}
S_2^{-1}  & 0 \\
0 & 0
\end{bmatrix}(N_{c,1})z+\underbrace{z^\transpose(N_{c,1})^\transpose Y(N_{c,1})z}_{\geq 0}\nonumber \\
& \geq  z^\transpose (N_{c,1})^\transpose\begin{bmatrix}
S_2^{-1}  & 0 \\
0 & 0
\end{bmatrix}(N_{c,1})z\nonumber\\
&=z^\transpose
\begin{bmatrix}
(N_{c,1}^1)^\transpose S_2^{-1} N_{c,1}^1 & 0\\
0 & 0\end{bmatrix}z.
\end{align}
Instead, $\lambda_1$ in \eqref{eq: lambda new} can be shown as 
\begin{align}\label{eq: lambda_1_final}
	\lambda_1=z^\transpose
	\begin{bmatrix}
	N_{1}^\transpose \left[ N_1\Sigma_{11}N_1^\transpose\right] ^{-1} N_{1} & 0\\
	0 & 0\end{bmatrix}z,
\end{align}
where we used the fact that $\widetilde N_1=N_1\Pi_1$. 

\smallskip 
\noindent \emph{\textit{Step 3 ($\lambda_c\geq \lambda_1$)}:}

For $\lambda_c\geq \lambda_1$ to hold true, it suffices to show the following:
\begin{align*}
(N_{c,1}^1)^\transpose S_2^{-1} N_{c,1}^1\geq N_{1}^\transpose \left[N_1\Sigma_{11} N_1^\transpose\right] ^{-1} N_{1}. 
\end{align*}
By invoking Proposition \ref{prop: null space containment}, we note that there exists a full row rank matrix $F_1$, such that $N_1=F_{1}N_{c,1}^{1}$. Since $F_1^\transpose$ is a full column rank matrix, we can define an invertible matrix $\widetilde{F}_1^\transpose \triangleq \left[ F_1^\transpose\, M_1^\transpose\right]$, where $M_1$ forms a basis for null space of $F_1$, such that the following holds 
\begin{align*}
S_2^{-1}&=\widetilde{F}_1^\transpose\left[\widetilde{F}_1S_2 \widetilde{F}_1^\transpose\right]^{-1}\widetilde{F}_1\\
&=\widetilde{F}_1^\transpose\begin{bmatrix}
F_1 S_2F_1^\transpose & F_1S_2M_1^\transpose \\
M_1 S_2F_1^\transpose & M_1 S_2M_1^\transpose 
\end{bmatrix}^{-1} \widetilde{F}_1.
\end{align*}
By invoking Schur's complement, it follows that 
\begin{align*}
\begin{bmatrix}
F_1 S_2F_1^\transpose & F_1S_2M_1^\transpose \\
M_1 S_2F_1^\transpose & M_1 S_2M_1^\transpose 
\end{bmatrix}^{-1}&=\begin{bmatrix}
\left( F_1 S_2F_1^\transpose\right)^{-1}  & 0 \\
0 & 0 
\end{bmatrix}+Y,
\end{align*}
where $Z\geq 0$. Hence,
\begin{align*}
(N_{c,1}^1)^\transpose S_2^{-1} N_{c,1}^1&=(\widetilde F_1N_{c,1}^1)^\transpose\begin{bmatrix}
	\left( F_1 S_2F_1^\transpose\right)^{-1} & 0 \\
	0 & 0 
	\end{bmatrix}	(\widetilde F_1N_{c,1}^1)\\
&\quad 
	+(\widetilde F_1N_{c,1}^1)^\transpose Z (\widetilde F_1N_{c,1}^1).
\end{align*}
By substituting $\widetilde F_1^\transpose=[F_1^\transpose \,M_1^\transpose]$ in the above expression, and rearranging the terms we have
\begin{align*}
(N_{c,1}^1)^\transpose S_2^{-1} N_{c,1}^1&=(F_1N_{c,1}^1)^\transpose\left( F_1 S_2F_1^\transpose\right)^{-1}(F_1N_{c,1}^1)\\
&\quad 
+(\widetilde F_1N_{c,1}^1)^\transpose Z (\widetilde F_1N_{c,1}^1).
\end{align*}
The required inequality follows by substituting $S_2=(N_{c,1}^1)\Sigma_{11}(N_{c,1}^1)^\transpose$ and $N_1=F_1N_{c,1}$, and recalling the fact that the sum of two positive semi definite matrices is greater than or equal to either of the matrices. \QED

\medskip 
\noindent{\it Proof of Lemma \ref{lma: PFD and PDD}:}

\smallskip
Let $\mathcal{E}_i$ be an event that the $i-$th local detector decides $H_1$ when the true hypothesis is $H_0$. Then, $P^{F}_i=\mathrm{Pr}\left[\mathcal{E}_i\right]$
. Let $\mathcal{E}_i^\complement$ be the complement of $\mathcal{E}_i$. Then, from \eqref{eq: PFD PDD definitions} it follows that 
\begin{align*}
P^F_d&=\mathrm{Pr}\left( \bigcup_{i=i}^N\mathcal{E}_i\right) =1-\mathrm{Pr}\left( \bigcap_{i=i}^N\mathcal{E}_i^\complement\right) \overset{(a)}{=}1-\prod_{i=1}^{N}\mathrm{Pr}\left( \mathcal{E}_i^\complement\right)\\ &=1-\prod_{i=1}^{N}\left(1-\mathrm{Pr}\left(\mathcal{E}_i\right)\right)=1-\prod_{i=1}^{N}\left( 1-P^F_i\right),
\end{align*} 
where for the $(a)$ we used the fact that the events $\mathcal{E}_i$ are mutually independent for all $i \in \{1,\ldots N\}$. To see this fact, notice that the event $\mathcal{E}_i$ is defined on  $\widetilde Y_i$ (see \eqref{eq: kernels}). Further, $\widetilde Y_i$ depends only on the deterministic attack signal $U^a_i$ and the noise vectors $V_i$ and $W_i$, but not on the interconnection signal $U_i$ (see \eqref{eq: expanded measurements local}). Now, by invoking the fact that noises variables across different subsystems are independent, it also follows that the events $\mathcal{E}_i$ are also mutually independent. Similar procedure will lead to the analogous expression for $P^D_d$ and hence, the details are omitted.  \QED

\medskip 
\noindent{\it Proof of Theorem \ref{thm: Pc greater than Pd}}: 

\smallskip 
Let $\mu_c=p_c+\lambda_c$ and $\sigma_c=\sqrt{2(p_c+2\lambda_c)}$, and assume that \eqref{eq: PDc greater than PDd} holds true. Then, from the monotonicity property of the CDF associated with the test statistic $\Lambda_c$, which follows $\chi^2(p_c,\lambda_c)$, we have the following inequality 
\begin{align*}
	\mathrm{Pr}\left[\Lambda_c\leq \tau_c\right]\leq \mathrm{Pr}\left[\Lambda_c\leq \mu_c-\sigma_c\sqrt{2N\ln\left(\frac{1}{1-P^D_\text{max}}\right) }\right].
\end{align*}
From the inequality \eqref{eq: tail bound two}, it now follows that 
\begin{align*}
\mathrm{Pr}\left[\Lambda_c\leq \tau_c\right]& \leq \exp\left(-N\ln\left( \frac{1}{1-P^D_\text{max}}\right) \right) \\
& = \exp\left(\ln\left(1-P^D_\text{max}\right)^N \right) \leq \prod_{i=1}^{N}\left(1-P^D_i\right), 
\end{align*}
where for the last inequality we used the fact that $P^D_i\leq P^D_\text{max}$ for all $i \in \{1,\ldots,N\}$. By using the above inequality and Lemma \ref{lma: PF and PD via Q-functions}, under hypothesis $H_1$, we have
\begin{align*}
	P^D_c=1-\mathrm{Pr}\left[\Lambda_c\leq \tau_c|H_1\right]\geq 1-\prod_{i=1}^{N}\left(1-P^D_i\right)=P^D_d. \quad \QED
	\end{align*}  
 
\smallskip 
\noindent{\it Proof of Theorem \ref{thm: Pd greater than Pc}}

\smallskip 
Let $\mu_c=p_c+\lambda_c$ and $\sigma_c=\sqrt{2(p_c+2\lambda_c)}$, and assume that \eqref{eq: PDd greater than PDc} holds true. Then, from the monotonicity property of the CDF associated with the test statistic $\Lambda_c$, which follows $\chi^2(p_c,\lambda_c)$, we have the following inequality 
\begin{align*}
\mathrm{Pr}\left[\Lambda_c\leq \tau_c\right]&\geq \mathrm{Pr}\left[\Lambda_c\leq \mu_c+\sigma_c\sqrt{2\ln\left(\frac{1}{1-(1-P^D_\text{min})^N}\right) }\right.\\
&\quad +\left. 2\ln\left(\frac{1}{1-(1-P^D_\text{min})^N}\right)\right].
\end{align*}
From the inequality \eqref{eq: tail bound one}, it now follows that 
\begin{align*}
\mathrm{Pr}\left[\Lambda_c\leq \tau_c\right]& \geq 1-\exp\left(-\ln\left( \frac{1}{1-(1-P^D_\text{min})^N}\right) \right), \\
& = 1-\exp\left(\ln\left(1-\left(1-P^D_\text{min}\right)^N\right) \right), \\
&\geq \prod_{i=1}^{N}\left(1-P^D_i\right)=1-\underbrace{\left[ 1-\prod_{i=1}^{N}\left(1-P^D_i\right)\right]}_{P^D_d}. 
\end{align*}
The result follows by substituting $P^D_c=1-\mathrm{Pr}\left[\Lambda_c\leq \tau_c|H_1\right]$ in the above inequality. \QED

\medskip 
\noindent{\it Proof of Theorem \ref{thm: central attack performance}}
	
\smallskip By recursively expanding the equality constraint of the optimization problem (P.2) we have 
\begin{align*}
	\begin{bmatrix}
	\overline{x}(1)\\
	\vdots \\
	\overline{x}(T)
	\end{bmatrix}&=\mathcal{A}x(0)+\mathcal{B}_aU^a
\end{align*}
By using the above identity, (P.2) can also be expressed as 
\begin{align*}
& \underset{U^a}{\text{max}}
& & \underbrace{\left[\mathcal{A}x(0)+\mathcal{B}_aU^a\right]^\transpose \left[\mathcal{A}x(0)+\mathcal{B}_aU^a\right]}_{f(U^a)}  \\
& \text{subject to}
& & (U^a)^\transpose  M_c (U^a) \leq \widetilde{\delta}_c.
\end{align*}
From the first-order necessary conditions \cite{EC-SZ:13} we now have

\begin{subequations}
	\begin{align}
\nabla  \left( f(U^*_c)-\gamma  (U^*_c)^\transpose  M_c (U^*_c) \right)&=0,\\
\gamma\left(\widetilde \delta_c-(U^*_c)^\transpose  M_c (U^*_c) \right)&=0 \label{eq: comp_slack},\\
\gamma&\geq 0,\\
(U^*_c)^\transpose  M_c (U^*_c) &\leq \widetilde{\delta}_c,
\end{align}
\end{subequations}
where the gradient $\nabla$ is with respect to $U^a$. 

\smallskip 
\noindent{\it Case (i)}: Suppose $(U^*_c)^\transpose  M_c (U^*_c) < \widetilde{\delta}_c$. Then $\gamma=0$ should hold true to ensure the complementarity slackness condition \eqref{eq: comp_slack}. Using these observations in the KKT conditions we now have $\nabla f(U^*_c)=0$. Further, since, $f(U^a)$ is a convex function of $U^a$, by evaluating the second derivative of $f(U^a)$ at $U^*_c$, it can be easily seen that the obtained $U^*_c$ results in minimum value of (P.2) rather than the maximum.Thus, for any $U^*_c$ of (P.2), the condition $(U^*_c)^\transpose  M_c (U^*_c) < \widetilde{\delta}_c$ cannot hold true. 

\smallskip 
\noindent{\it Case (ii)}: Suppose $	(U^*_c)^\transpose  M_c (U^*_c) = \widetilde{\delta}_c$. Then the KKT conditions can be simplified to the following: 
\begin{align*}
	\nabla  \left( f(U^*_c)-\gamma  (U^*_c)^\transpose  M_c (U^*_c) \right)&=0,\\
	(U^*_c)^\transpose  M_c (U^*_c) &= \widetilde{\delta}_c.
\end{align*}
The result now follows by evaluating the derivative on the left hand side of the first equality. \QED

\medskip 
\noindent{\it Proof of Lemma \ref{lma: central attack performance zero initial state}}

\smallskip 
By substituting $x(0)=0$ in \eqref{eq: optimal necessary one}, we note that any optimal attack $U^*_c$ is of the form $k\nu$, where $\nu$ is the generalized eigenvector of the pair $(\mathcal{B}^\transpose_a\mathcal{B}, M_c)$ \cite{DK:05}, and the scalar $k=\sqrt{\widetilde \delta_c/\nu^\transpose M_c \nu}$ is obtained from \eqref{eq: optimal necessary two}. Let $J_c$ be the optimal cost associated with an attack of the form $U^*_c=k\nu$. Then, 
\begin{align*}
	J_c&=(k\nu)^\transpose \mathcal{B}^\transpose_a\mathcal{B}_a (k\nu)=\gamma (k\nu)^\transpose M_c (k\nu)=\gamma \widetilde\delta_c, 
\end{align*}
where the first equality follows from the fact that the objective function  $\sum_{k=1}^{T}\overline{x}^\transpose(k)\overline x(k)$ in (P.2) can be expressed as $(U^*_c)^\transpose \mathcal{B}^\transpose_a\mathcal{B}_a (U^*_c)$, and the second equality follows from \eqref{eq: optimal necessary one}. Since $\nu$ is a generalized eigenvector of the pair $(\mathcal{B}^\transpose_a\mathcal{B}, M_c)$, it follows that $\gamma$ is the eigenvalue corresponding to $\nu$ and hence, $J_c$ is maximized when $\gamma$ is maximum, which is obtained for $v=v^\star$. The result follows since, $\gamma=\rho_\text{max}$, for $v=v^\star$. \QED

\medskip 
\begin{proposition}\label{prop: null space containment} Let $\mathcal{O}_i$ $\mathcal{F}_i^{(u)}$ be the observability and impulse response matrices defined in \eqref{eq: expanded measurements local}. Define 
	\begin{align}\label{eq: null spaces}
	\begin{split}
	\mathcal{N}^L_i&=\left\lbrace z: z^\transpose\begin{bmatrix} \mathcal{O}_i & \mathcal{F}^{(u)}_i
	\end{bmatrix}=0^\transpose\right\rbrace,\\
	\mathcal{N}^L_{c,i}&=\left\lbrace z: z^\transpose\mathcal{O}_{c,i}=0^\transpose\right\rbrace, \text{ and }\\
	\mathcal{N}^L_{c}&=\bigcup^{N}_{i=1}\mathcal{N}^L_{c,i}.
	\end{split}
	\end{align}
	where $\mathcal{O}_{c,i}=
	\begin{bmatrix}
	\left( C_{c,i}A\right)^\transpose  & \cdots & \left( C_{c,i} A^{T}\right)^\transpose
	\end{bmatrix}^\transpose$ and $C_{c,i}=\begin{bmatrix}
	0& \cdots & C_i & \cdots & 0
	\end{bmatrix}$. Then, $\mathcal{N}^L_i\subseteq \mathcal{N}^L_{c,i}\subseteq \mathcal{N}^L_{c}$, for all $i \in \{1,\ldots,N\}$.
\end{proposition}
\begin{proof}
	Without loss of generality, let $i=1$. By definition, the set inclusion 
	$\mathcal{N}^L_{c,1}\subseteq \mathcal{N}^L_{c}$ is trivial. For the other inclusion, consider the system defined in \eqref{eq: overall with attack} without the attack and noise, i.e., $x(k+1)=Ax(k)$. Let $x(k)=\begin{bmatrix}x_1^\transpose(k) & u_1^\transpose(k)\end{bmatrix}^\transpose$, where $x_1(k)$ and $u_1(k)$ are the state and the interconnection signal of Subsystem $1$. Also, let
	\begin{align}\label{eq: A decomposition}
	A=\begin{bmatrix}
	A_{11} & B_1\\
	\widetilde B_1 & \widetilde A_{11} 
	\end{bmatrix}. 
	\end{align}
	Notice that, $x(k+1)=Ax(k)$ can be decomposed as
	\begin{align}\label{eq: decomposed equations}
	\begin{split}
	x_1(k+1)&=A_{11}x_1(k)+B_1u_1(k), \\
	u_1(k+1)&=\widetilde A_{11}u_1(k)+\widetilde B_1x_1(k). 
	\end{split}
	\end{align}
	By letting $\widetilde C_1=\begin{bmatrix}
	C_1A_{11} & C_1B_1
	\end{bmatrix}$ and recursively expanding $x_1(k)$ using \eqref{eq: decomposed equations}, we have 
	\begin{align}\label{eq: impulse expansion}
	C_{c,1}A^{k}x(0)&=\begin{bmatrix}
	C_1 & 0
	\end{bmatrix}AA^{k-1}x(0)\nonumber\\
	&=\widetilde C_1A^{k-1}x(0)\nonumber \\
	&=\widetilde C_1\begin{bmatrix}
	x_1(k-1)\nonumber\\
	u_1(k-1)
	\end{bmatrix}\nonumber\\
	&=\widetilde C_1\begin{bmatrix}
	A_{11}^{k-1}x_1(0)+\sum_{j=0}^{k-2}A_{11}^{k-2-j}B_1u_1(j)\nonumber\\
	u_1(k-1)
	\end{bmatrix}\\
	&=C_1A_{11}^{k}x_1(0)+\sum_{j=0}^{k-1}C_1A_{11}^{k-1-j}B_1u_1(j),
	\end{align}
	where the second, third, and fourth equalities follows from \eqref{eq: A decomposition}, system $x(k+1)=Ax(k)$, and \eqref{eq: decomposed equations}, respectively. By recalling that
	$\mathcal{O}_{c,1}x(0)=\begin{bmatrix}
	\left(C_{c,1}A\right)^\transpose  &\cdots \left(C_{c,1}A^T\right)^\transpose
	\end{bmatrix}^\transpose x(0)$, it follows from \eqref{eq: impulse expansion} that 
	\begin{align*}
	\mathcal{O}_{c,1}x(0)=\mathcal{O}_{1}x_1(0)+\mathcal{F}^{(u)}_1\begin{bmatrix}
	u_1^\transpose(0) &
	\cdots &
	u_1^\transpose(T-1)
	\end{bmatrix}^\transpose. 
	\end{align*}
	Let $z$ be any vector such that $z^\transpose\begin{bmatrix} \mathcal{O}_1 & \mathcal{F}^{(u)}_1
	\end{bmatrix}=0^\transpose$. Then, $z$ also satisfies $z^\transpose\mathcal{O}_{c,1}=0^\transpose$. Thus, $\mathcal{N}^L_1\subseteq \mathcal{N}^L_{c,1}$. 
\end{proof}

\medskip 
\begin{lemma}{\bf {\it \textbf{(Upper bound on $P^D_d$)}}}\label{lma: ub PDD}
	Let $p_i$ and $\lambda_i$ be defined as in \eqref{eq: lambda 1}, and $\tau_i$ be defined as in \eqref{eq: test_stat_local}. Let $p_{\text{sum}}=\sum_{i=1}^Np_i$, $\lambda_{\text{sum}}=\sum_{i=1}^N\lambda_i$, and $\tau_{\text{min}}=\min\limits_{1 \leq i \leq N}\tau_i$. Then, 
	\begin{align*}
	P^D_d\leq \underbrace{\mathrm{Pr}\left[S_d>\tau_{\text{min}}\right]}_{\triangleq \overline{P}^D_d}, 
	\end{align*}where $S_d\sim \chi^2(p_{\text{sum}},\lambda_{\text{sum}})$. 
\end{lemma}

\begin{proof}
	Consider the following events:
	\begin{align*}
	\mathcal{V}_i&=\left\lbrace \widetilde Y_i^\transpose\Sigma^{-1}_i\widetilde Y_i\geq \tau_i\right\rbrace \text{ for all } i \in \{1,\ldots,N\}, \text{ and }\\
	\mathcal{V}&=\left\lbrace \sum_{i=1}^{N}\widetilde Y_i^\transpose\Sigma^{-1}_i\widetilde Y_i\geq \tau_\text{min}\right\rbrace, 
	\end{align*}
	where the event $\mathcal{V}_i$ is associated with the $i-$th local detector's threshold test. From the definition of the above events, it is easy to note that 
	$\bigcup_{i=1}^{N}\mathcal{V}_i\subseteq \mathcal{V}$. 
	By the monotonicity of the probability measures, it follows that
	\begin{align*}
	P^D_d\triangleq \mathrm{Pr}\left[\bigcup_{i=1}^{N}\mathcal{V}_i\left.\right\vert H_1\right]\leq \mathrm{Pr}\left[\mathcal{V}\left.\right\vert H_1\right]. 
	\end{align*}
	From the reproducibility property of the noncentral chi-squared distribution \cite{JNL-KS-BN:95}, it now follows that $\sum_{i=1}^{N}\widetilde Y_i^\transpose\Sigma^{-1}_i\widetilde Y_i$ equals $S_d$ in distribution and hence, $\mathrm{Pr}[\mathcal{V}|H_1]=\mathrm{Pr}[S_d> \tau_\text{min}]$.
\end{proof}

\medskip 
\begin{lemma}{\bf {\it \textbf{(Exponential bounds on the tails of $\chi^2(p,\lambda)$)}}}
	Let $Y \sim \chi^2(p,\lambda)$, $\mu=p+\lambda$, $\sigma=\sqrt{2(p+2\lambda)}$. For all $x>0$,
	\begin{subequations}
		\begin{align}
		\mathrm{Pr}\left[Y\geq \mu+\sigma\sqrt{2x}+2x\right]&\leq \exp(-x) \label{eq: tail bound one}\\
		\mathrm{Pr}\left[Y\leq \mu-\sigma\sqrt{2x}\right]&\leq \exp(-x) \label{eq: tail bound two}
		\end{align}
	\end{subequations}
\end{lemma}
\begin{proof}
	See \cite{birge}.
\end{proof}

\end{document}